\documentclass[12pt]{amsart}

\usepackage{amsmath,amsthm,amssymb,mathrsfs} 

\usepackage{graphicx}
\usepackage{subcaption}
\usepackage{comment}
\usepackage{enumerate}

\usepackage{yfonts} 

\usepackage{mathtools} 
\mathtoolsset{showonlyrefs}

\usepackage[pdftex,%
plainpages=false,%
pdfpagelabels,%
hypertexnames=true,
colorlinks=true,%
linkcolor=black,%
citecolor=black,%
bookmarksopen=true,%
bookmarksopenlevel=3,%
bookmarksnumbered=true]{hyperref}

\newtheorem{theorem}{Theorem}[section]
\newtheorem{proposition}[theorem]{Proposition}
\newtheorem{lemma}[theorem]{Lemma}
\newtheorem{corollary}[theorem]{Corollary}
\newtheorem{problem}{Problem}
\newtheorem{conj}{Conjecture}

\theoremstyle{definition}
\newtheorem{definition}[theorem]{Definition}
\newtheorem{remark}[theorem]{Remark}
\newtheorem{exmp}[theorem]{Example}

\renewcommand{\epsilon}{\varepsilon}
\renewcommand{\phi}{\varphi}

\renewcommand\Re{\operatorname{Re}}
\renewcommand\Im{\operatorname{Im}}

\newcommand{\ie}{i.e.\@,\@ }

\newcommand{\etal}{et al.\@}
\newcommand{\cf}{cf.\@ }

\newcommand{\stft}{short-time Fourier transform}

\newcommand{\tf}{time-frequency}

\newcommand{\R}{\mathbb{R}}
\newcommand{\Rd}{\mathbb{R}^d}
\newcommand{\Rtd}{\mathbb{R}^{2d}}

\newcommand{\Z}{\mathbb{Z}}

\newcommand{\N}{\mathbb{N}}

\newcommand{\Cf}{\mathbb{C}}
\newcommand{\Qf}{\mathbb{Q}}
\newcommand{\K}{\mathbb{K}}

\newcommand{\B}{\mathcal{B}}
\newcommand{\F}{\mathcal{F}}

\newcommand{\A}{\mathcal{A}}
\newcommand{\Hf}{\mathcal{H}}

\newcommand{\Bf}{\textfrak{B}}

\newcommand{\Sc}{\mathcal{S}}

\newcommand{\I}{\mathcal{I}}

\newcommand{\PB}{\B /_\sim}
\newcommand{\Ap}{\A_\Phi}
\newcommand{\Cp}{C_\Phi}

\newcommand*\dif{\mathop{}\!\mathrm{d}}

\makeatletter
\DeclareRobustCommand\widecheck[1]{{\mathpalette\@widecheck{#1}}}
\def\@widecheck#1#2{%
    \setbox\z@\hbox{\m@th$#1#2$}%
    \setbox\tw@\hbox{\m@th$#1%
       \widehat{%
          \vrule\@width\z@\@height\ht\z@
          \vrule\@height\z@\@width\wd\z@}$}%
    \dp\tw@-\ht\z@
    \@tempdima\ht\z@ \advance\@tempdima2\ht\tw@ \divide\@tempdima\thr@@
    \setbox\tw@\hbox{%
       \raise\@tempdima\hbox{\scalebox{1}[-1]{\lower\@tempdima\box
\tw@}}}%
    {\ooalign{\box\tw@ \cr \box\z@}}}
\makeatother

\newcommand{\Fhat}{\widehat{\phantom{x}}}  

\DeclareMathOperator{\vspan}{span}

\DeclareMathOperator{\im}{ran}

\DeclareMathOperator{\rank}{rank}
\DeclareMathOperator{\diag}{diag}

\DeclareMathOperator{\supp}{supp}

\DeclareMathOperator{\opt}{opt}

\DeclareMathOperator{\trace}{trace}

\begin{document}

\title{Phase Retrieval: Uniqueness and Stability}

\author{Philipp Grohs}
 \address{Faculty of Mathematics \\
 University of Vienna \\
 Oskar-Morgenstern-Platz 1 \\
 A-1090 Vienna, Austria}
\email{philipp.grohs@univie.ac.at}

\author{Sarah Koppensteiner}
 \address{Faculty of Mathematics \\
 University of Vienna \\
 Oskar-Morgenstern-Platz 1 \\
 A-1090 Vienna, Austria}
\email{sarah.koppensteiner@univie.ac.at}

\author{Martin Rathmair}
 \address{Faculty of Mathematics \\
 University of Vienna \\
 Oskar-Morgenstern-Platz 1 \\
 A-1090 Vienna, Austria}
\email{martin.rathmair@univie.ac.at}

\begin{abstract}
  The problem of phase retrieval, i.e., the problem of recovering a
  function from the magnitudes of its Fourier transform, naturally
  arises in various fields of physics, such as astronomy, radar,
  speech recognition, quantum mechanics, and, perhaps most prominently,
  diffraction imaging. The mathematical study of phase retrieval
  problems possesses a long history with a number of beautiful and
  deep results drawing from different mathematical fields, such as
  harmonic analysis, complex analysis, and Riemannian geometry. The
  present paper aims to present a summary of some of these results
  with an emphasis on recent activities. In particular we aim to
  summarize our current understanding of uniqueness and stability
  properties of phase retrieval problems.
\end{abstract}

\maketitle

\section{Introduction}\label{sec:intro}

The problem of phase retrieval, i.e., the problem of recovering a
function from the magnitudes of its Fourier transform, naturally
arises in various fields of physics, such as astronomy
\cite{dainty1987astronomy}, radar \cite{Jaming1999}, speech
recognition \cite{Rabiner1993FSR}, and quantum mechanics
\cite{pauli1946allgemeinen}.  The most prominent example, however, is
diffraction imaging, where in a basic experiment an object is placed
in front of a laser which emits coherent electromagnetic radiation.
The object interacts with the incident wave in a diffractive manner,
creating a new wave front, which is described by Kirchhoff's
diffraction equation.  An adequate approximation of the resulting wave
front in the far field is given by the Fraunhofer diffraction
equation, which essentially states that the wave front in a plane at a
sufficiently large distance from the object is given by the Fourier
transform (with appropriate spatial scaling) of the function
representing the object; cf.~\cite{goodman2005introduction}
for an introduction to diffraction theory.

The aim in diffractive imaging is to determine the object from
measurements of the diffracted wave.  This objective is seriously
impeded by the fact that measurement devices usually are only capable
of capturing the intensities, and a loss of phase information takes
place.  Reconstructing the object from the far field diffraction
intensities, the so-called diffraction pattern, therefore requires one
to solve the \emph{Fourier phase retrieval} problem
\begin{center}
Given $|\hat{f}|$, find $f$ (up to trivial ambiguities).
\end{center}
The name ``phase retrieval'' accounts for the fact that recovery of the
phase of $\hat{f}$ is equivalent to recovering $f$ itself.

In microscopy a lens is employed to essentially invert the Fourier
transformation and create the image of the object.  While this is
possible in the case of visible light, which has a wavelength of
approximately $10^{-7} \text{m}$, lenses which perform this task are
not available for waves of much shorter wavelength (e.g. for x-rays
with a wavelength in the range between $10^{-8}\text{m}$ and
$10^{-11}\text{m}$).  Since the spatial resolution of the optical
system is proportional to the wavelength of light the direct approach
using lenses can only achieve a certain level of resolution.  In order
to obtain high resolution it
is necessary to compute the image from the diffraction pattern.

Determining objects from diffraction patterns---and therefore the
question of phase retrieval---for the first time became relevant when
Max von Laue discovered in 1912 that x-rays are diffracted when
interacting with crystals, an insight for which he would be awarded
the Nobel Prize in Physics two years later.  The discovery of
this phenomenon launched the field of x-ray crystallography.
Crystallography seeks to determine the atomic and molecular structure
of a crystal, i.e., a material whose atoms are arranged in a periodic
fashion.  In the diffraction pattern the periodicity of the
crystalline sample manifests itself in the form of strong peaks (Bragg
peaks) lying on the so-called reciprocal lattice;
cf.~\cite{Millane:90}.  From the position and the intensities of these
peaks chrystallographers can deduce the electron density of the
crystal.  Over the course of the past century the methods of x-ray
crystallography have developed into the most powerful tool for
analyzing the atomic structure of various materials and have enabled
scientists to achieve breakthrough results in different fields such as
chemistry, medicine, biology, physics, and the material sciences.  This is
highlighted by the fact that more than a dozen Nobel Prizes have been
awarded for work involving x-ray crystallography, the discovery of the
double helix structure of DNA \cite{watson1953molecular} being just
one example.  For an exhaustive introduction to x-ray crystallography
the interested reader may have a look at
\cite{hammond2001basics,ladd2014structure}.

In 1980 it was proposed by David Sayre \cite{imprcoph1980} to extend
the approach of x-ray crystallography to noncrystalline specimens.
Almost twenty years later, facilitated by the development of new
powerful x-ray sources, Sayre \etal \cite{miao1999extending} for the
first time successfully reconstructed the image of a sample with
resolution at nanometer scale from its x-ray diffraction pattern. 
This approach is nowadays known as Coherent Diffraction
Imaging (CDI).  The process consists of two principal steps.  First,
the acquisition of one or multiple diffraction patterns, and second,
processing the diffraction patterns in order to obtain the image of
the sample, which is usually done by applying iterative phase
retrieval algorithms.  Plenty of CDI methods have been developed in
recent years and have been employed to great success in physics,
biology, and chemistry.  See \cite{Miao530,Shechtman2015phase} for very
recent overviews of CDI methods, for their limitations and their
achievements in various applications, and for algorithmic phase
retrieval methods in diffraction imaging.

Even though the quest for recovering lost phase information has been
omnipresent in physics for more than a century now, the phase
retrieval problem has only very recently---with a few
exceptions---started receiving great attention by the mathematics
community.  One notable exception is the work of Herbert Hauptman
beginning in the 1950s.  The direct methods developed by Hauptman
\cite{hauptman86direct}, together with Jerome Karle have been applied
with great success to determine the structure of many crystals.
In 1985 Hauptman and Karle were awarded the Nobel Prize in Chemistry.

As a second significant exception we mention the work of Joseph
Rosenblatt from the 1980s \cite{rosenblatt1984phase}, where the problem
of phase retrieval
from Fourier magnitudes is studied in great generality.

Phase retrieval in the most general formulation is concerned with
reconstructing a function $f$ in a space $\mathcal{X}$ from the
phaseless information of some transform of $f$.  The operator
describing the transform, which will be denoted by $T$, is mapping
elements of $\mathcal{X}$ into another space $\mathcal{Y}$ of either
real- or complex-valued functions and is usually linear,
i.e.,
$$T:\mathcal{X}\rightarrow\mathcal{Y}.$$
Furthermore, $T$ is
usually nicely invertible, which means that
$T:\mathcal{X}\rightarrow \text{ran} T$ has a bounded inverse.

In order to have a concrete example in mind one may think of
$\mathcal{X}=\mathcal{Y}=L^2(\Rd)$ and $T=\mathcal{F}$, the Fourier
transform operator.  In this case it is well known that $T$ is a
unitary map.

Under the above assumptions the linear measurement process does not
introduce a loss of information.  However, the situation changes
significantly if the phase information of the transform is absent.
The problem arises of studying the obviously nonlinear mapping
$$
\mathcal{A}:f \mapsto |Tf|, \quad f\in \mathcal{X},
$$
and its invertibility properties.  Well-posedness in the sense of
Hadamard of an inverse problem associated with $f\mapsto \mathcal{A}f$
requires
\begin{enumerate}[(1)]
\item\label{wp1} \emph{existence} of a solution, i.e., $\mathcal{A}$
  to be surjective;
\item\label{wp2} \emph{uniqueness}, i.e., $\mathcal{A}$ to be
  injective; and
\item\label{wp3} \emph{stability}, meaning that the solution
  continuously depends on the data, i.e., $\mathcal{A}^{-1}$ to be
  continuous.
\end{enumerate}
For the problem of phase retrieval, condition~\eqref{wp1} amounts to
identifying the image of the operator $\mathcal{A}$.  The question is
often of minor importance compared to \eqref{wp2} and \eqref{wp3} as
it is simply assumed that the input data arise from the measurement
process described by $\mathcal{A}$.

Provided that $\mathcal{X}$ is a vector space---excluding trivial
cases---$\mathcal{A}$ is not injective due to the simple observation
that
\begin{equation}\label{eq:trivambig}
  \mathcal{A}f = \mathcal{A}(cf),\quad f\in\mathcal{X}, ~|c|=1.
\end{equation}
Further ambiguities may occur, such as translations in the Fourier
example but also less trivial ones.  The first key question in the
mathematical analysis of a phase retrieval problem is to identify all
ambiguities.  Depending on the context a particular source of
ambiguities is either classified as trivial or as severe.  If there
exist severe ambiguities the phase retrieval problem is hopeless as
there exist different objects yielding identical measurements.  If on
the other hand all occurring ambiguities are considered trivial, $f$
and $g$ may be identified ($f\sim g$) whenever
$\mathcal{A}f=\mathcal{A}g$.  Let
$\tilde{\mathcal{X}}=\mathcal{X}/ \sim$ denote the quotient set.  Then---by definition---$\mathcal{A}$ is injective as mapping acting on
$\tilde{\mathcal{X}}$ and
uniqueness in this new sense is ensured.

In order to study stability, $\tilde{\mathcal{X}}$ has to be endowed
with a reasonable topology first.  In the case where $\mathcal{X}$ is a normed
space and the only ambiguities occurring are of the type shown in
\eqref{eq:trivambig}, usually the quotient metric
$$
d([f]_\sim,[g]_\sim):=\inf_{|c|=1} \|f-cg\|
$$
is used. If there are other ambiguities, a suitable choice may be less
obvious.

Beyond determining whether the mapping $\mathcal{A}$ on
$\tilde{\mathcal{X}}$ is continuously invertible further continuity
properties of the inverse are often studied such as (local) Lipschitz
continuity.

If there are nontrivial ambiguities, i.e., if injectivity is not
attained after identifying all trivial ambiguities or if the inverse
is not continuous, one or both of the following measures may be taken
in order to render the phase retrieval problem well-posed:
\begin{enumerate}[(A)]
\item Restriction of $\mathcal{A}$: The restriction
  $\mathcal{A}:\tilde{\mathcal{X}}'\rightarrow
  \mathcal{A}(\tilde{\mathcal{X}}')$,
  where $\tilde{\mathcal{X}}'\subset \tilde{\mathcal{X}}$ is
  eventually injective (has a continuous inverse) if
  $\tilde{\mathcal{X}}'$ is chosen sufficiently small,
  $\tilde{\mathcal{X}}'$ consisting of a single element being the extremal, trivial example. \\
  Restriction of $\mathcal{A}$ to a smaller domain can be understood
  as imposing additional structural assumptions on the function $f$ to
  be reconstructed.  In applications of the phase retrieval problem from
  Fourier measurements, for instance, it is typically sensible to
  demand that $f$ be nonnegative, as other functions do not hold a
  physical meaning.
\item Modification of $T$: The idea is to suitably modify $T$ in order
  to soften the setback which is suffered by the subsequent removal of
  the phase information.\\
  In the case of the Fourier phase retrieval problem this can be achieved
  by applying several different manipulations of $f$ before
  computation of the Fourier transform, e.g., using
 \begin{equation}\label{mod:ptycho}
 T'f:=\left(\widehat{fg_1},\ldots,\widehat{fg_m}\right)
 \end{equation}
 for known functions $g_1,\ldots,g_m$ instead of $Tf=\hat{f}$.  In the
 context of diffraction imaging this approach is common practice, as a
 physical system which produces measurements
 $|T'f|=\left(|\widehat{fg_1}|,\ldots, |\widehat{fg_m}|\right)$ can
 often be implemented.  In ptychography---a concept proposed by Walter
 Hoppe in the 1960s \cite{Hoppe1969beugung}---different sections of an
 object are illuminated one after another and the object is to be
 reconstructed from several diffraction patterns.  For suitable,
 localized window functions $g_1,\ldots,g_m$, \eqref{mod:ptycho} serves
 as a reasonable mathematical model.\\
 As a second example let us mention holography, invented by Dennis
 Gabor in 1947 \cite{gabor1948new}. In holography the diffracted waves
 interfere with the wave field of a known object.  This idea amounts
 to an additive distortion of the wave field $T'f:=\widehat{f+g}$,
 where $g$ is a known reference wave.
\end{enumerate}

 To the best of our knowledge, these ideas (the restriction and
 modification approach that is) have been systematically implemented
 for the first time in a series of papers by
 Jaming~\cite{Jaming1999,jaming14uniqueness}.
 
 When studying a concrete phase retrieval problem with an application
 in the background it is useful to keep in mind that often there is a
 certain degree of freedom in the way the measurements are
 acquired.  For instance, in diffraction imaging there is the
 fundamental observation that the wave in the object plane and the
 wave in the far field are connected in terms of the Fourier
 transform.  However there are many different options in how to
 generate one or several diffraction patterns.  Instead of viewing a
 phase retrieval problem as the analysis of a fixed operator
 $\mathcal{A}$ one may as well include the question of how to design
 the measurement process in order to get a well-posed problem.

 Beyond the question of well-posedness it is desirable to provide a
 method that recovers a function $f$ (at least the equivalence class
 $[f]_\sim$) from the observed measurements $\mathcal{A}f$.  Such a
 method could be an explicit expression of the inverse of
 $\mathcal{A}$.  Mostly the aim of coming up with an explicit
 expression is rather hopeless.  In practice iterative algorithms are
 employed, which serve as approximate inverses of the measurement
 mapping $\mathcal{A}$.  A framework based on iterative projections
 which is often simple to implement and has proved to be very flexible
 was introduced by Gerchberg and Saxton~\cite{gerchberg71} and was
 later extended by Fienup~\cite{fienup82}.  These methods have been
 employed to great empirical success, but due to the absence of
 convexity there is no guarantee of convergence.  In recent years
 Candes, Strohmer, and Voroninski~\cite{candes13} have studied phase
 retrieval in a random setup and proposed an algorithm which provably
 recovers $f$ with high probability.

 Phase retrieval problems have been studied in a rich variety of
 shapes.  They can be distinguished between finite- and
 infinite-dimensional as well as between discrete and continuous phase
 retrieval problems.  Furthermore phase retrieval problems differ in
 what kind of measurements are considered, i.e., the choice of the
 operator $T$.  The most common choice is that $T$ involves some sort
 of Fourier transform
 \cite{akutowicz1956determination,walther1963question,
   hofstetter1964construction,alaifari2016stable,Mallat2015}.
 Moreover there is a huge body of research in the more abstract
 setting of frames, where it is assumed that $T$ is induced by a frame
 \cite{MR2224902,MR3202304,MR3554699,MR3656501}.  Phase retrieval
 problems where the quantity of interest is assumed to arise as the
 solution of certain differential equations have also been studied
 \cite{jaming17,jaming:hal-01514078}.

 It is the aim of the present paper to present an overview of a
 selection of the aforementioned developments.  In Section
 \ref{sec:abstractpr} we summarize our current understanding of
 abstract phase retrieval problems,
 that is, without precisely specifying the nature of the observed measurements.
 
 Then we specialize to phase retrieval problems arising from (masked
 or windowed) Fourier transform measurements.  The finite-dimensional
 case is considered in section \ref{sec:finitepr}, and the continuous
 infinite-dimensional setting in section \ref{sec:infpr}.  We present
 several (well-known and also new) results on uniqueness and stability
 of the corresponding phase retrieval problems.
 
 Finally, we believe that phase retrieval offers researchers a
 unique combination of beautiful and deep mathematics as well as very
 concrete physical applications. It is our hope to convey some of our
 enthusiasm for this topic to the reader.


\section{Abstract Phase Retrieval}\label{sec:abstractpr}

From an abstract point of view, Fourier phase retrieval lends itself
to the following interpretation: Of a function $f$, we are given the
absolute values of measurements given by bounded linear functionals.
In the case of Fourier phase retrieval, the family of linear
functionals are just the pointwise evaluation of the Fourier transform
$\{ f \mapsto \hat{f}(x) : x \in \Rd \}$.

With this interpretation in mind, we can phrase the phase retrieval
problem in a more abstract way.  Throughout this section let $\B$
denote a Banach space over $\K \in \{\R, \Cf\}$ and $\B'$ its
topological dual space.  Furthermore, let $\Lambda$ be a not
necessarily countable index set.  For a family of bounded linear
functionals
$\Phi := \{ \phi_\lambda : \lambda \in \Lambda\} \subseteq \B'$, we
define the operator of phaseless measurements by
\begin{equation*}
  \Ap f := ( | \langle f, \phi_\lambda \rangle |)_{\lambda \in \Lambda} \,,
\end{equation*}
where $\langle\, . \, , \, . \, \rangle$ denotes the dual pairing.
Due to the linearity, it is clear that $\Ap (c f) = \Ap f$
for phase factors $|c| = 1$.  We therefore introduce the
equivalence relation $c f \sim f$ and say $\Phi$ \emph{does phase
  retrieval} if the mapping
\begin{equation*}
  \Ap : \PB \to \R_+^{\Lambda}
\end{equation*}
is injective.

\subsection{Injectivity}

Suppose $\Phi:= \{\phi_\lambda : \lambda \in \Lambda\} \subseteq \B'$
is a family of bounded linear functionals and $S \subseteq \Lambda$. We then write
$\Phi_{S}:= \{\phi_\lambda : \lambda \in S\} \subseteq \Phi$.  For a
linear subspace $V$ of $\B'$, let
$V_\perp := \{ f \in \B : \langle f, v \rangle = 0 \quad \forall v \in
V \}$ denote the \emph{annihilator of $V$ in $\B$}.

\begin{definition}
  The family $\Phi \subseteq \B'$ satisfies the \emph{complement
    property} in $\B$ if we have $(\vspan \Phi_S)_\perp = \{0\}$ or
  $(\vspan \Phi_{\Lambda \setminus S})_\perp = \{0\}$ for every
  $S \subseteq \Lambda$.
\end{definition}

Then the complement property is necessary for $\Ap$ to be
injective.  In the real case, it is even sufficient.  

\begin{theorem}
  \label{sec:injectivity-CPnecessary}
  Let $\B$ be a Banach space over $\K \in \{\R, \Cf\}$ and
  $\Phi \subseteq \B'$ a family of bounded linear functionals.  Then
  the following hold:
  \begin{enumerate}
  \item If $\Ap$ is injective, then $\Phi$ satisfies the
    complement property.
  \item If $\K = \R$ and $\Phi$ satisfies the complement property,
    then $\Ap$ is injective.
  \end{enumerate}
\end{theorem}

Theorem~\ref{sec:injectivity-CPnecessary} has quite a history.  It was
first stated for finite dimensions in Balan, Casazza, and
Edidin~\cite{MR2224902}.  The arguments for the complex case should
have been given more care.  Bandeira \etal~\cite{MR3202304} spotted
this oversight and gave an alternative proof for the complex case in
finite dimensions.  In doing so, they produced a series of
characterizations for injectivity in finite dimensions.  Moreover,
they had the crucial insight for stability of phase retrieval by
introducing a ``numerical version'' of the complement property (see
section~\ref{sec:stability}).

Ultimately, only a minor correction was necessary to repair Balan
\etal\!'s proof and the same arguments also work in infinite
dimensions.  This is the proof we present below, which can also be
found in \cite{MR3656501,MR3440085,MR3554699}.

\begin{proof}
  (i) Let $\Ap$ be injective for
  $\Phi= \{\phi_\lambda : \lambda \in \Lambda\}$ and
  $S \subseteq \Lambda$ arbitrary.  Assume that there is a nonzero
  $f \in (\vspan \Phi_S)_\perp$ and let
  $h \in (\vspan \Phi_{\Lambda \setminus S})_\perp$.  We have to show
  that $h = 0$.  First note that
    \begin{equation*}
      | \langle f \pm h, \phi_\lambda \rangle |^2 = 
      | \langle f, \phi_\lambda \rangle |^2  \pm 
      \underbrace {2 \Re( \langle f, \phi_\lambda\rangle 
        \langle h, \phi_\lambda\rangle)}_{= 0} + | \langle h, \phi_\lambda \rangle |^2 
      \qquad \forall \lambda \in \Lambda \,. 
  \end{equation*}
  Hence $\Ap(f+h) = \Ap (f-h)$.  As $\Ap$ is assumed to be injective,
  there exists a phase factor $|c|=1$ such that $f+h = c (f-h)$. Since
  $f \neq 0$ we have $c \neq -1$ and then
  \begin{equation*}
    h = \frac{c - 1}{1+ c} f \in (\vspan \Phi_S)_\perp 
    \cap (\vspan \Phi_{\Lambda \setminus S})_\perp  \,,
  \end{equation*}
  which implies that $\Ap h = 0$. Now the injectivity of $\Ap$ implies
  $h=0$ as expected.




  (ii) 
  Suppose $\Ap$ is not injective, this means that there exist
  $f,h \in \B$ such that $\Ap f = \Ap h$.  Since
  $\Phi= \{\phi_\lambda : \lambda \in \Lambda\}$ consists of
  real-valued linear functionals, the signed measurements of $f$ and
  $h$ with respect to $\Phi$ can only differ by a factor of
  $c = -1$.  We therefore consider the following partition of the
  index set $\Lambda$:  Let
  $S:= \{ \lambda \in \Lambda : \langle f, \phi_\lambda \rangle =
  \langle h, \phi_\lambda \rangle\}$;
  then
  $ \Lambda \setminus S = \{ \lambda \in \Lambda : \langle f,
  \phi_\lambda \rangle = - \langle h, \phi_\lambda \rangle\}$.

  Consequently, $f-h \in (\vspan \Phi_S)_\perp$ and
  $f+h \in (\vspan \Phi_{\Lambda \setminus S})_\perp$.  But by
  assumption at least one of those annihilators consists only of $0$.
  Hence $f=h$ or $f=-h$ and therefore $\Ap$ is injective.
\end{proof}




For the Paley--Wiener space
$PW^{p,b}_{\R}:= \{f \in L^p(\R, \R) : \supp \hat f \subseteq [ - b /
2, b /2 ]\}$
($1< p < \infty$) of real-valued band-limited functions, one can show
that the complement property holds for families of point-evaluations
$\Phi=\{\delta_{\lambda} : \lambda \in \Lambda \}$ if the sampling
rate exceeds twice the critical density~\cite{MR3656501}.  Since
$PW^{p,b}_{\R}$ is a real-valued Banach space, this implies that phase
retrieval is possible. 

For complex Banach spaces, the complement property is not sufficient.
Hence other methods need to be employed to study injectivity. For
Fourier-type measurements, these tools often come
from complex analysis (see sections~\ref{sec:finitepr} and \ref{sec:infpr}). 


We now turn to the finite-dimensional case.  The complement property
implies that $\Phi \subseteq \K^d$ needs to span the whole space and
must be overcomplete for phase retrieval to be possible.  In other
words, $\Phi$ must be a frame.  

In the remainder of this section, we state necessary and sufficient
conditions on the number of frame elements of $\Phi$ to do phase
retrieval.  The first result is an easy consequence of the complement
property.

\begin{corollary}
  \label{sec:injectivity-NecessaryR}
  If $N < 2d - 1$, then $\Ap$ cannot be injective for any
  family $\Phi \subseteq \K^d$ with $N$ elements.
\end{corollary}

\begin{proof}
  We partition $\Phi$ into two sets
  $\Phi_S, \Phi_{ \Lambda \setminus S}$ with at most $d-1$ elements.
  This yields $\vspan \Phi_S \neq \K^d$ and
  $\vspan \Phi_{ \Lambda \setminus S}\neq \K^d$, clearly violating the
  complement property.
\end{proof}

For $\K=\R$, the converse statement also holds for ``almost all''
frames.  To make this more precise, we need some terminology of
algebraic geometry.

An \emph{algebraic variety} in $\K^d$ is the common zero set of
finitely many polynomials in $\K[x_1,\dots, x_d]$.  By defining
algebraic varieties in $\K^d$ as closed, we obtain the \emph{Zariski
  topology}.  Note that this topology is coarser than the Euclidean
topology on $\K^d$, meaning that every Zariski-open set is also open
with respect to the Euclidean topology.  Furthermore, nonempty
Zariski-open sets are dense with respect to the Euclidean topology and
have full Lebesgue measure in $\K^d$~\cite{MR3202304,MR3303679}.

We say a \emph{generic point} in $\K^d$ satisfies a certain property,
if there exists a nonempty Zariski-open set with this property.  By
the above, this means that if a certain property holds for a generic
point, it holds for almost all points in $\K^d$.


Now we identify a frame $\Phi \subseteq \K^d$ of $N$ elements with a
$d \times N$ matrix of full rank.  Hence the set of frames with $N$
elements in $\K^d$, \ie the set of matrices of full rank in
$\K^{d \times N}$, is a nonempty Zariski-open set and it makes sense
to study generic points within the set of frames.  We call those
generic points \emph{generic frames}.


The following theorem is due to Balan, Casazza, and Edidin~\cite{MR2224902}.
Together with Corollary~\ref{sec:injectivity-NecessaryR}, it (almost)
characterizes the injectivity of phase retrieval in $\R^d$.

\begin{theorem}
  If $N \geq 2d - 1$, then $\Ap$ is injective for a generic frame
  $\Phi \subseteq \R^d$ with $N$ elements.
\end{theorem}

For phase retrieval in $\Cf^d$, Bandeira \etal~\cite{MR3202304}
conjectured an analogous characterization with $4d-4$ being the
critical number of frame elements.  They also gave a proof in
dimensions $d=2,3$.  Conca \etal~\cite{MR3303679} (see also
\cite{Kiraly2014}) proved the following theorem, confirming the
sufficient part of the $(4d-4)$-Conjecture.

\begin{theorem}
  Let $d \geq 2$. If $N \geq 4d - 4$, then $\Ap$ is injective for a
  generic frame $\Phi \subseteq \Cf^d$ with $N$ elements.
\end{theorem}

Conversely, a frame in $\Cf^d$ with $N < 4d - 4$ elements does not
allow phase retrieval in dimensions $d= 2^k +1$~\cite{MR3303679}.  But
the $(4d-4)$-Conjecture does not hold in general:
Vinzant~\cite{Vinzant2015} gave an example of a frame with
$11 = 4d - 5$ elements in $\Cf^4$ which does phase retrieval.  For
necessary lower bounds in general dimension, we refer the interested
reader to Wang and Xu~\cite{Wang2016}.  A more in-depth account of the
history of necessary and sufficient bounds for phase retrieval in
$\Cf^d$ can be found in~\cite{MR3440085}.  Furthermore, Bodmann and
Hammen~\cite{Bodmann13,Bodmann14} developed concrete algorithms and
error bounds for phase retrieval with low-redundancy frames.

\subsection{Stability}
\label{sec:stability}

Once the question of injectivity is answered positively, the question
of stability arises.  Stability refers to the continuity of the
operator $\Ap^{-1}: \im \Ap \to \PB$.  To this end, we need to
introduce a topology on $\PB$ and find a suitable Banach space $\Bf$
with $\im \Ap \subseteq \Bf \subseteq \K^{\Lambda}$.  The natural
choice for $\PB$ is the quotient metric
\begin{equation*}
  d(f,h):= \inf_{|c|=1} \|f-c h \|_{\B} \,.
\end{equation*}

The analysis space for frames in separable Hilbert spaces is the
sequence space $\ell^2(\Lambda)$.  We will consider the stability of
phase retrieval for continuous Banach frames in this section.  There,
the appropriate generalization of $\ell^2(\Lambda)$ is an
``admissible'' Banach space $\Bf$ such that the range of the
coefficient operator
\begin{equation*}
  \Cp f := (  \langle f, \phi_\lambda \rangle )_{\lambda \in \Lambda} 
\end{equation*}
is contained in $\Bf$.



\begin{definition}
  Let $\Lambda$ be a $\sigma$-compact topological space. A Banach
  space $\Bf \subseteq \K^\Lambda$ is called
  \emph{admissible} if it satisfies the following properties:
  
  \begin{enumerate}
  \item The indicator function $\chi_{K}$ of every compact set
    $K \subseteq \Lambda$ satisfies
    $\| \chi_{K} \|_{\Bf} < \infty$.
  \item The Banach space $\Bf$ is \emph{solid}; this means that
    $ \|w\|_{\Bf} \leq \|z\|_{\Bf}$ whenever
    ${|w(\lambda)| \leq  |z(\lambda)|}$ for all $\lambda \in \Lambda$.
  \item The elements of $\Bf$ with compact support are dense in $\Bf$.
  \end{enumerate}
\end{definition}

These properties are quite reasonable.  Indeed, all $L^p$-spaces for
$1 \leq p < \infty$ are admissible Banach spaces and $L^\infty$
violates only the last point unless $\Lambda$ is already compact.

Now we are in a position to define stability of phase retrieval
precisely.

\begin{definition}
  Let $\Phi \subseteq \B'$ be a family of bounded linear functionals
  and $\Bf$ and admissible Banach space such that
  $C_{\Phi} : \B \to \Bf$.  We say
  that the phase retrieval of $\Phi$ is \emph{stable} (with respect to
  $\Bf$) if there exist constants
  $0<\alpha \leq \beta < \infty$ such that
  \begin{equation}
    \label{eq:stability}
    \alpha d(f,h) \leq \|\Ap(f)-\Ap(h)\|_{\Bf} \leq \beta d(f,h) \qquad \forall f,h \in \B
  \end{equation}
  Moreover, let $\alpha_{\opt}(\Phi), \beta_{\opt}(\Phi)$ denote the optimal
  lower and upper Lipschitz bound respectively.
\end{definition}

\begin{definition}
  Suppose that
  $\Phi := \{\phi_\lambda : \lambda \in \Lambda\} \subseteq \B'$ is a
  family of bounded linear functionals such that
  $\lambda \mapsto \phi_\lambda$ is continuous.  We call $\Phi$ a
  \emph{continuous Banach frame} if there exists an admissible Banach
  space such that the following is satisfied:
  \begin{enumerate}
  \item There exist positive constants $0 < A \leq B < \infty$ such that
    \begin{equation}
      \label{eq:frameineq}
      A \|f\|_{\B} \leq \| \Cp f \|_{\Bf} \leq B \|f\|_{\B}
      \qquad \forall f \in \B \,.
    \end{equation}
    Moreover, let $A_{\opt} (\Phi), B_{\opt} (\Phi)$ denote the
    optimal constants satisfying \eqref{eq:frameineq}.
  \item There exists a continuous operator $R: \Bf \to \B$, the
    so-called \emph{reconstruction operator}, satisfying
    \begin{equation*}
      R\Cp f = f \qquad \forall f \in \B \,.
    \end{equation*}
  \end{enumerate}
\end{definition}

The requirement for $\Phi$ to be a frame is a natural one.  In fact,
if $\Cp$ maps into an admissible Banach space, the solidity implies
$\| \Ap f \|_{\Bf} = \| \Cp f \|_{\Bf}$.  Hence, stability in the
sense of \eqref{eq:stability} implies the frame inequality
\eqref{eq:frameineq} by taking $h = 0$.  For the upper inequalities,
we even have equivalence.  

\begin{proposition}
  If $\Phi \subseteq \B'$ is a family of continuous linear functionals
  such that $\Cp$ maps into an admissible Banach space, then
  $\beta_{\opt} = B_{\opt}$.
\end{proposition}

Again the solidity of the admissible Banach space plays an integral
role in the proof.  As the rest follows from straightforward
estimates, we omit the proof and refer the interested reader to
\cite{MR3656501,MR3554699}.

The remainder of the section deals with the lower inequality in
\eqref{eq:stability}.  We start by mentioning an interesting result
about the continuity of the inverse operator $\Ap^{-1}$, which can be
regarded as a weaker form of stability.

\begin{theorem}
  \label{sec:stability-weakstability}
  Let $\Phi \subseteq \B'$ be a continuous Banach frame and $\Ap$
  injective.  Then $\Ap^{-1}$ is continuous on the range of $\Ap$.
\end{theorem}

\begin{proof}[Proof idea]
  We need to show that the convergence of the image sequence
  $\Ap f_n \to \Ap f$ in $\Bf$ implies the convergence of $f_n \to f$
  in $\B$.  The idea is to link the convergence of $\Ap f_n$ to the
  convergence of the signed measurements $\Cp f_n$.  This is the
  technical and lengthy part of the proof, and we refer the interested
  reader to \cite{MR3656501} for the details.  Once this relation is
  established, one can use the continuous reconstruction operator
  $R$ to obtain $f_n \to f$.
\end{proof}


As an easy consequence of Theorem~\ref{sec:stability-weakstability},
we obtain stability of phase retrieval in finite-dimensional Banach
spaces.

\begin{theorem}
  \label{sec:stability-finitestability}
  Let $\B$ be a finite-dimensional Banach space.  If $\Phi$ is a frame
  that does phase retrieval, then $\Ap$ has a lower Lipschitz bound
  $\alpha_{\opt} >0$.
\end{theorem}

\begin{proof}
  Note that the existence of a positive lower Lipschitz bound
  $\alpha_{\opt} >0$ in \eqref{eq:stability} is equivalent to
  $\Ap^{-1}:\im \Ap \to \PB $ being Lipschitz continuous with constant
  $L= \alpha_{\opt}^{-1}$.

  By Theorem~\ref{sec:stability-weakstability}, the inverse $\Ap^{-1}$
  is continuous on $\im \Ap$.  Since $\B$ is finite-di\-men\-si\-onal,
  the closed unit ball $B(0,1)$ is compact, and therefore $\Ap^{-1}$
  is uniformly continuous on $\im \Ap \cap B(0,1)$.  By using the
  scaling invariance of $\Ap^{-1}$ and playing everything back into
  the unit ball $B(0,1)$, the Lipschitz continuity follows in a series
  of straightforward estimates.
\end{proof}

The result of Theorem~\ref{sec:stability-finitestability} was proved
first for the real case in \cite{MR3323113,MR3202304}.
Cahill, Casazza, and Daubechies~\cite{MR3554699} gave a proof for the complex case.
The proof above is from \cite{MR3656501}.

For their proof of stability in finite dimensions, Bandeira
\etal~\cite{MR3202304} introduced the following ``numerical'' version
of the complement property, which relates to stability as the
complement property relates to injectivity.

\begin{definition}
  The family
  $\Phi \subseteq \B'$
  satisfies the \emph{$\sigma$-strong complement property} in $\B$ if
  there exists a $\sigma > 0$ such that 
  \begin{equation}
    \label{eq:SCP}
    \max\{A_{\opt} (\Phi_S), A_{\opt} (\Phi_{\Lambda \setminus S})\} \geq \sigma  
    \qquad \forall S \subseteq \Lambda \,.
  \end{equation}
  Moreover, let $\sigma_{\opt}(\Phi)$ denote the supremum over all
  $\sigma >0$ satisfying \eqref{eq:SCP}.
\end{definition}

\begin{theorem}
  \label{sec:stability-SCPnecessary}
  Let $\B$ be a Banach space over $\K \in \{\R, \Cf \}$ and
  $\Phi \subseteq \B'$ a continuous Banach frame.  Then there
  exists a constant $C>0$ such that
  \begin{equation*}
    \alpha_{\opt} \leq C \sigma_{\opt} \,.
  \end{equation*}
  In the real case, the constant is $C=2$.  For the complex case, the
  constant can be chosen $C= 2 B_{\opt}/A_{\opt}$.
\end{theorem}

\begin{remark}
  For the real case, one can also show that
  $\sigma_{\opt} \leq C \alpha_{\opt}$ for some $C>0$.  This implies
  that the $\sigma$-strong complement property is not only necessary,
  but also sufficient for stability in real Banach spaces.  In this
  sense, it mirrors the behavior of the complement property.

  Unfortunately, the sufficiency cannot be exploited for (global)
  stability: On one hand, phase retrieval is always stable in finite
  dimensions by Theorem~\ref{sec:stability-finitestability} and on the
  other hand, we will see that the $\sigma$-strong complement property
  can never hold in infinite dimensions.
\end{remark}

\begin{proof}
  Let $\sigma > \sigma_{\opt}$.  Then there exist a subset
  $S \subseteq \Lambda$ and $f,h \in \B$ with
  $\|f\|_{\B} = \|h\|_{\B} = 1$ such that
  \begin{equation}\label{ineq:cpsigma}
    \|C_{\Phi_S} f \|_{\Bf} < \sigma \quad \text{ and } \quad
    \|C_{\Phi_{\Lambda \setminus S}} h \|_{\Bf} < \sigma \,.
  \end{equation}

  Now set $x:= f+h$ and $y:=f-h$.  Due to the solidity of $\Bf$, we
  obtain
  \begin{align*}
    \| \Ap(x) - \Ap(y)\|_{\B}
    & \leq \| ( |\langle x, \phi_{\lambda}\rangle| -
      |\langle y, \phi_{\lambda}\rangle|)_{\lambda \in S} \|_{\Bf}
      + \| ( |\langle x, \phi_{\lambda}\rangle| - |\langle y,
      \phi_{\lambda}\rangle|)_{\lambda \in \Lambda \setminus S} \|_{\Bf} \\
    & \leq 2 \|C_{\Phi_S} f \|_{\Bf} + 2 \|C_{\Phi_{\Lambda \setminus S}} h \|_{\Bf} \\
    &  \leq 4 \sigma \, ,
  \end{align*}
  where we used the reverse triangle inequality in the second line.

  By definition of $\alpha_{\opt}$, we conclude
  \begin{equation*}
    \alpha_{\opt} d(x,y) \leq \| \Ap(x) - \Ap(y)\|_{\B} \leq 4 \sigma \,.
  \end{equation*}

  In the real case, we are done since
  $$d(x,y) = \min\{ \| x+y\|_{\B}, \|x-y\|_{\B}\} = 2 \min\{\|f\|_{\B},
  \|h\|_{\B}\} = 2.$$

  The complex case proves to be more difficult.  A series of
  elementary estimates are necessary to bound $d(x,y)$ away
  from zero.  We refer the interested reader to the original
  article~\cite{MR3656501}.
\end{proof}


\begin{remark}\label{rem:locstab}
  The computations in the proof of Theorem
  \ref{sec:stability-SCPnecessary} also yield an estimate on local
  stability constants.  More precisely, suppose a fixed
  $x\in \mathcal{B}$ can be decomposed according to $x=f+h$ such that
  $\|f\|_{\B}\asymp 1$, $\|h\|_{\B}\asymp 1$ and that
  \eqref{ineq:cpsigma} holds for $\sigma \ll 1$.  Then there exists
  $y\in\mathcal{B}$ such that
  \begin{equation*}
    \|\Ap(x)-\Ap(y)\|_{\mathcal{B}} \lesssim \sigma \quad \text{and} 
    \quad d(x,y)\gtrsim 1 \,.    
  \end{equation*}
  Thus, $x$ and $y$ yield similar measurements even though they are
  very different from each other.
\end{remark}

Theorem~\ref{sec:stability-SCPnecessary} implies that the
$\sigma$-strong complement property is necessary for stability.
Bandeira \etal~\cite{MR3202304} gave a proof of this for the real case
and conjectured the complex case, which was proved in~\cite{MR3656501}.

For finite dimensions, phase retrieval is always stable by
Theorem~\ref{sec:stability-finitestability}.  In particular, the
$\sigma$-strong complement property is satisfied.  In infinite
dimensions, we will see that continuous Banach frames cannot satisfy
the $\sigma$-strong complement property, hence phase retrieval is
always unstable in this case.  To show this, we follow
\cite{MR3656501} and prove an intermediate result, which is
interesting in its own right.  It states that there cannot exist
continuous Banach frames in infinite dimensions with compact index set
$\Lambda$.

\begin{proposition}
  \label{sec:stability-1}
  Suppose $\B$ is an infinite-dimensional Banach space and $\Lambda$ a
  compact index set.  Then any family
  $\Phi := \{\phi_\lambda : \lambda \in \Lambda\} \subseteq \B'$ with
  continuous mapping $\lambda \mapsto \phi_{\lambda}$ fails to satisfy
  the lower frame inequality.  This means that for every $\epsilon > 0$ there
  exists an $f \in \B$ such that
  \begin{equation*}
    \|C_{\Phi}f\|_{\Bf} < \epsilon \|f\|_{\B} \,.
  \end{equation*}
\end{proposition}

\begin{proof}
  Let $\epsilon > 0$.  By continuity of the mapping
  $\lambda \mapsto \phi_{\lambda}$, there exists for every
  $\lambda \in \Lambda$ an open neighborhood $U_\lambda$ such that
  \begin{equation*}
    \|\phi_\omega - \phi_\lambda \|_{\B'} < \frac{\epsilon}{\|\chi_\Lambda\|_{\Bf}} 
    \qquad \forall \omega \in U_\lambda \,.
  \end{equation*}

  Since $\Lambda$ is compact, the open covering
  $\{U_\lambda : \lambda \in \Lambda \}$ admits a finite subcover
  $\{U_{\lambda_1}, \dots , \allowbreak U_{\lambda_N}\}$.  Now set
  $U_1:=U_{\lambda_1}$ and
  $U_j:= U_{\lambda_j} \setminus \bigcup_{k=1}^{j-1}U_k$ for
  $j=2, \dots, N$ to obtain a partition of $\Lambda$ which satisfies the following
  for all $j = 1, \dots, N$:
  \begin{equation*}
    \|\phi_\lambda - \phi_{\lambda_j} \|_{\B'} <
    \frac{\epsilon}{\|\chi_\Lambda\|_{\Bf}} \qquad \forall \lambda \in U_j \,.
  \end{equation*}

  Clearly, we have
  \begin{equation*}
    |\langle f, \phi_\lambda \rangle | \leq |\langle f, \phi_{\lambda_j} \rangle |
    + |\langle f, \phi_\lambda - \phi_{\lambda_j} \rangle |
  \end{equation*}
  for all $j = 1, \dots N$.  After multiplication with the
  characteristic function $\chi_{U_j}$ and summing over $j$, we obtain
  \begin{align*}
    \Ap f(\lambda) &= \sum_{j=1}^N |\langle f, \phi_\lambda \rangle | \chi_{U_j}(\lambda) \\
    & \leq  \sum_{j=1}^N |\langle f, \phi_{\lambda_j} \rangle |  \chi_{U_j}(\lambda)
      +\sum_{j=1}^N |\langle f, \phi_\lambda - \phi_{\lambda_j} \rangle | \chi_{U_j}(\lambda) \\
    & < \sum_{j=1}^N |\langle f, \phi_{\lambda_j} \rangle |  \chi_{U_j}(\lambda) 
      + \frac{\epsilon \|f\|_{\B}}{\|\chi_\Lambda\|_{\Bf}}   \chi_\Lambda(\lambda) \,.
  \end{align*}
  
  Now the solidity of $\Bf$ implies
  \begin{equation*}
    \|C_{\Phi}f\|_{\Bf} =  \|\Ap f\|_{\Bf} <
    \sum_{j=1}^N |\langle f, \phi_{\lambda_j} \rangle |  \|\chi_{U_j} \|_{\Bf}
    + \epsilon \|f\|_{\B}
  \end{equation*}
  for all $f \in \B \setminus \{0\}$.  Since $\B$ is infinite-dimensional, there
  exists a nonzero $f_0 \in \B$ such that
  $\langle f_0, \phi_{\lambda_j} \rangle =0 $ for all
  $j = 1, \dots, N$.  Consequently, the sum on the right-hand side
  vanishes for $f_0$ and we obtain the claim.
\end{proof}

\begin{theorem}
  \label{sec:stability-NoSCP}
  Let $\B$ be an infinite-dimensional Banach space over
  $\K \in \{\R, \Cf\}$ and $\Phi \subseteq \B'$ a continuous Banach
  frame.  Then $\Phi$ does not satisfy the $\sigma$-strong complement
  property.  
\end{theorem}

\begin{proof}
  We need to show that the $\sigma$-strong complement property is not
  satisfied.  This means that for every $\epsilon > 0$ we can find a
  subset $S \subseteq \Lambda$ and $f, h \in \B$ such that
  \begin{equation*}
    \|C_{\Phi_S} f \|_{\Bf} < \epsilon \|f\|_{\B} \quad \text{ and } \quad
    \|C_{\Phi_{\Lambda \setminus S}} h \|_{\Bf} < \epsilon \|h\|_{\B} \,.
  \end{equation*}

  We start with an arbitrary $f \in \B$ with $\|f\|_{\B} =1$.  Since
  $\Bf$ is an admissible Banach space where compact elements are
  dense, there exists a nested sequence of compact subsets
  $K_n \subseteq K_{n+1}$ with $\bigcup_{n \in \N} K_n = \Lambda$ such
  that
  \begin{equation*}
    \| C_{\Phi} f - C_{\Phi}f \cdot \chi_{K_n} \|_{\Bf} \to 0 \qquad \text{as } n \to \infty \,.
  \end{equation*}

  Hence, there exists a $K_N$ such that
  \begin{equation*}
    \| C_{\Phi} f - C_{\Phi}f \cdot \chi_{K_N} \|_{\Bf} < \epsilon \,.
  \end{equation*}
  Setting $S:= \Lambda \setminus K_N$, we obtain $\|C_{\Phi_S} f \|_{\Bf}<\epsilon \|f\|_{\B}$.

  On the other hand, we can use Theorem~\ref{sec:stability-1} for the
  compact set $\Lambda \setminus S = K_N$ to find an $h \in \B$ such
  that
  \begin{equation*}
    \|C_{\Phi_{\Lambda \setminus S}} h \|_{\Bf} < \epsilon \| h \|_{\B} \,.
  \end{equation*}
\end{proof}

\begin{corollary}
  \label{sec:stability-CorIns}
  Let $\B$ be an infinite-dimensional Banach space over~ 
  $\K \in \{\R, \Cf\}$ and $\Phi \subseteq \B'$ a continuous Banach
  frame.  Then $\Phi$ cannot do stable phase retrieval.  This means
  that for every $\epsilon > 0$, there exist $f, h \in \B$ with
  $\|\Ap(f)-\Ap(h)\|_{\Bf} < \epsilon$ but $d(f,h) \geq 1$.
\end{corollary}

\begin{proof}
  This is an immediate consequence of the fact that the
  $\sigma$-strong complement property is necessary for stability by
  Theorem~\ref{sec:stability-SCPnecessary}, but continuous Banach
  frames in infinite dimensions cannot satisfy it by
  Theorem~\ref{sec:stability-NoSCP}.
\end{proof}

\begin{remark}
  Phase retrieval in infinite dimensions cannot be stable for
  continuous Banach frames by Corollary~\ref{sec:stability-CorIns}.
  On the other hand, Theorem~\ref{sec:stability-finitestability}
  states that it is always stable in finite dimensions.  The natural
  question that arises is the following: Suppose $V_n \subseteq \B$ is
  a sequence of finite-dimensional subspaces and let $\alpha(V_n)$
  denote the stability constant for the subspace $V_n$ in
  \eqref{eq:stability}.  How fast does the stability constant
  $\alpha(V_n)$ degenerate as the dimension increases?

  It turns out, this can be rather rapidly: Cahill, Casazza, and
  Daubechies~\cite{MR3554699} considered subspaces of increasing
  dimension in the Paley--Wiener space
  and showed that the stability constant degrades exponentially fast
  in the dimension.  Even worse degeneration can be observed for the
  short-time Fourier transform with Gaussian window on $L^2(\R)$:
  Alaifari and one of the authors~\cite{alaifari2018gabor} constructed
  a sequence of subspaces whose stability constant degrades
  quadratically exponentially in the dimension.


\end{remark}



\section{Finite Dimensional Phase Retrieval}\label{sec:finitepr}
This section is devoted to phase retrieval from Fourier measurements
in the finite-dimensional setting.  The first emphasis lies on
identifying ambiguities for phase retrieval from phaseless
discrete time Fourier transform (DTFT) measurements.  The second main
focus lies on discussing various strategies to remove ambiguities, and
yield well-posed reconstruction problems.  These strategies include
priors such as assuming sparsity of the signals to be reconstructed or
tweaking of the measurement process, for example by increasing the
number of measurements and/or by introducing randomness.
\subsection{The classical Fourier Phase Retrieval Problem}
\label{subsec:fourier}
In the following we will discuss the problem of recovering a signal
from its phaseless Fourier transform.
We consider multidimensional discrete signals.  This means that for
$n\in \N^d$, a discrete signal is a complex-valued function on
  $$
  J_n:= \{0,\ldots,n_1-1\}\times\dots\times\{0,\ldots,n_d-1\}\,.
  $$ 
\begin{definition}
  The \emph{discrete-time Fourier transform} $\hat{x}$ of a
  discrete signal $x=(x_j)_{j\in J_n} \in \Cf^{J_n}$ is defined by
  \begin{equation*}\label{def:dtft}
    \hat{x}(\omega):= \sum_{j\in J_n} x_j e^{-2\pi i j\cdot \omega / n} ,
    \qquad \omega\in\Rd \,, 
  \end{equation*}
  where the normalization
  $\omega / n := (\omega_1 / n_1 , \dots, \omega_d / n_d)$ is
  understood componentwise and
  $j \cdot \omega := \sum_{k=1}^d j_k \omega_k$ denotes the inner
  product on $\Rd$.
\end{definition}

The problem of Fourier phase retrieval can now be stated as follows.
\begin{problem}[Fourier phase retrieval, discrete]\label{prob:fprdiscrete}
  Recover $x\in \Cf^{J_n}$ from $|\hat{x}|$.
\end{problem}
\begin{remark}
  For $x\in \Cf^{J_n}$ the squared modulus of its DTFT $|\hat{x}|^2$ is a
  trigonometric polynomial and is uniquely defined by its values on a
  suitable, finite sampling set $\Omega\subseteq \Rd$.  The problem of
  recovering $x$ from the full Fourier magnitude
  $|\hat{x}(\omega)|, ~\omega \in \Rd$ is therefore equivalent to the
  problem of recovering $x$ from finitely many samples of the Fourier
  magnitude $|\hat{x}(\omega)|, ~\omega\in \Omega$.
\end{remark}
The goal is to characterize all ambiguous solutions of Problem
\ref{prob:fprdiscrete} for a given signal $x\in \Cf^{J_n}$.  Before we
do so let us draw the attention of the reader to Fienup's paper
\cite{Fienup1978reconstruction} from the 1970s where the following
observation is made:
\begin{quotation}
  Experimental results suggest that the uniqueness problem is severe
  for one-dimensional objects but may not be severe for complicated
  two-di\-men\-si\-onal objects.
\end{quotation}
Within this section we will give a rigorous explanation of this phenomenon.

Before we identify ambiguities, we have to explain what it means to reflect and translate a signal
$x\in \Cf^{J_n}$.  We define the \emph{reflection operator} $R$ on $\Cf^{J_n}$ by
$$
(Rx)_j = x_{-j \pmod{n}} \qquad \forall j\in J_n
$$
and the \emph{translation operator} $T_\tau$ for $\tau \in \Z^d$ by
$$
(T_\tau x)_j = x_{j-\tau \pmod{n}} \qquad \forall j\in J_n \,,
$$
where the modulo operation is to be understood componentwise.  
Similarly the conjugation operation will be understood componentwise.  
For $z\in\Cf^d$ and $j\in \Z^d$ we will write
$z^{-1}:=(z_1^{-1},\ldots,z_d^{-1})$ and
$z^j:=z_1^{j_1}\cdot \ldots \cdot z_d^{j_d}$ for short.
\begin{proposition}\label{prop:fourierambig}
  Let $x\in \Cf^{J_n}$.  Then each of the following choices of $y$
  yields the same Fourier magnitudes as $x$, i.e.,
  $|\hat{y}|=|\hat{x}|$:
 \begin{enumerate}
  \item $y=cx$ for $|c|=1$;
  \item $y=T_\tau x$ for $\tau\in \Z^d$;
  \item $y=\overline{Rx}$.
 \end{enumerate}
\end{proposition}
\begin{proof}
  The statement follows from (i) linearity of the Fourier transform,
  (ii) translation amounts to modulation in the Fourier domain and
  (iii) reflection and conjugation amounts to conjugation in the
  Fourier domain.
\end{proof} 
The ambiguities described in Proposition \ref{prop:fourierambig},
as well as combinations thereof, are considered trivial.  By
identifying trivial ambiguities an equivalence relation $\sim$ is
introduced on $\Cf^{J_n}$, i.e.,
$$
x\sim y \quad\Leftrightarrow \quad y=cT_\tau \overline{Rx} ~~\text{or}~~ y=c T_\tau x, 
\quad \text{where}~ \tau \in \Z^d, ~|c|=1 \,.
$$
To determine all ambiguities we will study the so-called $Z$-transform.
  \begin{definition}\label{def:ztrafo}
   For $x\in \Cf^{J_n}$ the \emph{$Z$-transform} is defined by
   $$
   X(z):=(\mathcal{Z}x)(z):= \sum_{j\in J_n} x_j z^j \qquad \forall z\in \Cf^d\,.
   $$
  \end{definition}
  The question of uniqueness of Problem \ref{prob:fprdiscrete} is closely
  connected to whether the $Z$-transform has a nontrivial
  factorization, as we shall see.
\begin{definition}
  A polynomial $p$ of one or several variables is called
  \emph{reducible} if there exist nonconstant polynomials $q$ and $r$
  such that $p=q\cdot r$. Otherwise $p$ is called \emph{irreducible}.
\end{definition}
In what follows, let $p(z)=\sum_j c_j z^j$ denote a multivariate
polynomial.  Its degree $\deg(p)\in \N_0^d$ is defined with respect to
each coordinate, \ie
\begin{equation*}
  \deg(p)_k=\max\{j_k: ~c_j\neq 0\} \qquad k=1,\ldots,d \,.
\end{equation*}
Later we will need to consider the mapping
$z\mapsto \overline{p(\bar{z}^{-1})}$.  Clearly, its singularities can
be removed by multiplication with a suitable monomial.  Indeed,
\begin{equation*}
  \label{eq:defq}
  p^*(z):=\overline{p(\bar{z}^{-1})}\cdot z^{\deg(p)}
\end{equation*}
is again a polynomial.  Finally, let $\nu(p)\in\N_0^d$ denote the
largest exponent (com\-po\-nent-wise) such that $z^{\nu(p)}$ is a divisor
of $p$.  Thus there exists a unique polynomial $p_0$ such that
\begin{equation*}
  \label{eq:facp}
  p(z)=z^{\nu(p)}p_0(z)\,.
\end{equation*}
To shed some more light onto these concepts we consider a concrete example.
\begin{exmp}
 Let us consider the polynomial $p$ on $\Cf^2$ defined by 
 $$p(z_1,z_2)=z_1+iz_1^3 z_2^2 = z_1(1+iz_1^2z_2^2).$$
 We see that $\deg(p)=(3,2)$ and that $\nu(p)=(1,0)$.
 Moreover, 
 $$
 p_0(z_1,z_2)=1+iz_1^2z_2^2,
 $$
 and 
 $$
 p^*(z)=\overline{p(\bar{z}^{-1})}\cdot z^{\deg(p)} = (z_1^{-1}-iz_1^{-3}z_2^{-2})z_1^3 z_2^2 = z_1^2z_2^2-i.
 $$
\end{exmp}
It is not difficult to verify that for any polynomial $p\neq 0$ it holds that
\begin{equation*}\label{eq:nupstar}
 \nu(p^*)=0,
\end{equation*}
and that 
\begin{equation*}\label{eq:p0irpstarir}
 p_0 \,\text{is irreducible if and only if}\,\, p^*\,\text{is irreducible}.
\end{equation*}

The following theorem characterizes all ambiguities of the discrete
Fourier phase retrieval problem.

Note that multiplication of the Z-transform $X$ of $x$ by a unimodular
factor $\gamma$ by linearity corresponds to multiplication of the $x$
itself.  Multiplication of $z^\tau$ for $\tau \in \Z^d$ corresponds to
translation in the signal domain, and flipping the $Z$-transform,
i.e., passing over to $\overline{X(\bar{z}^{-1})}$, amounts to
reflection and conjugation in the signal domain.
\begin{theorem}
  \label{thm:discreterecubile}
  Let $x,y\in\Cf^{J_n}$ and let $X,Y$ denote their respective
  $Z$-transforms.  Then $|\hat{x}|=|\hat{y}|$ if and only if there
  exist a factorization $Y=Y_1\cdot Y_2$, a constant $\gamma$ with
  $|\gamma|=1$, and $\tau\in \Z^d$ such that
  \begin{equation}
    \label{eq:ambigztrafo}
    X(z)=\gamma z^\tau \cdot Y_1(z) \cdot \overline{Y_2(\bar{z}^{-1})} \,.
  \end{equation}
\end{theorem}
\begin{proof}
  First we show the necessity of the statement.  Suppose $y$ is an
  ambiguous solution with respect to $x$.  By definition
  $X(z)=\sum_{j\in J_n} x_j z^j$ and thus, using the notation
  \begin{equation*}
    e^{-2\pi i \omega/n}=\left(e^{-2\pi i \omega_1/ n_1},\ldots,e^{-2\pi i \omega_d / n_d}\right),
    \qquad \omega\in \Rd \,,    
  \end{equation*}
  we observe that $X(e^{-2\pi i \omega/n})=\hat{x}(\omega)$.  For the
  squared magnitude of the Fourier transform it therefore holds that
  \begin{equation*}
    \label{eqXY}
    |\hat{x}(\omega)|^2= X(e^{-2\pi i \omega/n}) \cdot \overline{X(e^{-2\pi i \omega/n})} 
    = X(e^{-2\pi i \omega}) \cdot \overline{X(\overline{e^{-2\pi i \omega}}^{-1})} \,,
  \end{equation*}
  where conjugation and the reciprocal are to be understood
  component-wise.  By the assumption that $|\hat{x}|=|\hat{y}|$ and by
  analytic continuation, we obtain
  \begin{equation}
    \label{eq:XXbar}
    X(z)\cdot \overline{X(\bar{z}^{-1})} = Y(z)\cdot\overline{Y(\bar{z}^{-1})}
    \qquad \forall z \in \Cf^d \setminus \{0\}\,.
  \end{equation}
  Now factorize $X$ and $Y$ into irreducible polynomials:
  \begin{equation*}
    X(z)=z^{\nu(X)} \prod_{i=1}^L p_i(z) \quad\text{and}
    \quad Y(z)=z^{\nu(Y)} \prod_{i=1}^{L'} p'_i(z)\,.
  \end{equation*}
  After multiplying both sides of \eqref{eq:XXbar} by $z^n$, we
  obtain the equality
  \begin{equation*}
    z^{n-\sum\limits_{i=1}^{L}\deg(p_i)} \cdot \prod_{i=1}^L p_i(z) \cdot 
    \prod_{i=1}^{L} z^{\deg(p_i)} \overline{p_i(\bar{z}^{-1})} = 
    z^{n-\sum\limits_{i=1}^{L'}\deg(p'_i)} \cdot \prod_{i=1}^{L'} p'_i(z) 
    \cdot \prod_{i=1}^{L'} z^{\deg(p'_i)} \overline{p'_i(\bar{z}^{-1})} \,.
  \end{equation*}
  Since $p_i$ is irreducible it follows that 
  $p_i^*(z)=z^{\deg(p_i)} \overline{p_i(\bar{z}^{-1})}$ is irreducible. Moreover we have that $\nu(p_i^*)=0$.
  Obviously the same arguments can be applied to 
  $(p'_i)^*(z)=z^{\deg(p'_i)} \overline{p'_i(\bar{z}^{-1})}$.

  By uniqueness of the factorization it follows that
  \begin{equation}\label{eq:equalityppprime}
    \prod_{i=1}^L p_i \cdot \prod_{i=1}^{L} p_i^*
    = \prod_{i=1}^{L'} p_i'\cdot \prod_{i=1}^{L'} (p'_i)^*
  \end{equation}
  and that $L=L'$. 
  Now let $I$ be a maximal subset of $\{1,\ldots,L\}$ such that
  $\prod_{i\in I} p_i$ divides $\prod_{i=1}^{L} p'_i$ and let
  $J:=\{1,\ldots,L\}\setminus I$.  Without loss of generality
  (w.l.o.g.) we assume that $I=\{1,\ldots,l\}$ with $l\le L$ and that
  $\prod_{i\le l} p_i$ divides $\prod_{i\le l} p_i'$ (this can be
  achieved by permutation of the index sets). Due to irreducibility it
  must hold that
  \begin{equation}\label{eq:firstlfactors}
   \prod_{i\le l} p_i=a \prod_{i\le l} p_i'
  \end{equation}
  for suitable nonzero constant $a$, and, consequently that 
  \begin{equation}\label{eq:firstlfactorsstar}
   \prod_{i\le l} p_i^*=b \prod_{i\le l} (p_i')^*,
  \end{equation}
    where $b$ is another nonzero constant. Use \eqref{eq:firstlfactors} and \eqref{eq:firstlfactorsstar} and cancel \eqref{eq:equalityppprime} 
    by the respective factors to obtain that 
    \begin{equation*}
     \prod_{i>l} p_i \prod_{i>l}p_i^* = c \prod_{i>l}p_i' \prod_{i>l} (p_i')^*,
    \end{equation*}
for a constant $c\neq0$.
From the maximality of $I$ it follows that $\prod_{i>l}p_i$ divides $\prod_{i>l} (p_i')^*$, and thus that 
$\prod_{i>l}p_i^*$ divides $\prod_{i>l}p_i'$. Therefore there exists $d\neq 0$ such that 
\begin{equation*}
 \prod_{i>l}p_i^* = d \prod_{i>l}p_i'.
\end{equation*}
Hence we get that  
\begin{equation*}
    X(z)=
    z^{\nu(X)}  \prod_{i\le l}p_i(z)  \prod_{i>l}p_i(z)
    =z^{\nu(X)} \left( a\prod_{i\le l} p'_i\right)  \cdot \left( d \prod_{i\in J'} (p'_i)^*\right).
  \end{equation*}
  Note that $|ad| =1$, since
  \begin{equation*}
    |X(1)| = |ad| \prod_{i\in I'} |p'_i(1)|  \cdot  \prod_{i\in J'} |(p'_i)^*(1)|
    = |ad| \prod_{i\in I'} |p'_i(1)|  \cdot  \prod_{i\in J'} |p'_i(1)| = |ad| \cdot |Y(1)| \,.
  \end{equation*}
  Consequently, we obtain for suitable $m\in \Z^d$
  and $\gamma=ad$ the factorization
  $$
  X(z)=\gamma z^m \cdot Y_1(z) \cdot \overline{Y_2(\bar{z}^{-1})}\,,
  $$
  with $Y_1:=\prod_{i\le l} p'_i$ and $Y_2:=\prod_{i>l} p'_i$.\\

  For the sufficiency let $X$ be a polynomial of the form
  \eqref{eq:ambigztrafo}.  Then
  \begin{align*}
    |\hat{x}(\omega)|^2&=X(e^{-2\pi i \omega / n})\cdot
                         \overline{X(e^{-2\pi i \omega/n})} \\
                       & = Y_1(e^{-2\pi i \omega/n}) \cdot
                         \overline{Y_1(e^{-2\pi i \omega/n})} \cdot Y_2(e^{-2\pi i
                         \omega/n}) \cdot \overline{Y_2(e^{-2\pi i \omega/n})} =
                         |\hat{y}(\omega)|^2\,.
  \end{align*}
\end{proof}
For $x$ to have nontrivial ambiguities it is therefore necessary that
its $Z$-transform $X$ be reducible.  Note that this is not sufficient
in general, as the factors of $X$ may possess symmetry properties such
that a flipping does not introduce nontrivial ambiguities.
Nevertheless, this observation yields an upper bound on the number of
ambiguous solutions for $x\in \Cf^{J_n}$ denoted by
\begin{equation*}
  \label{def:noa}
  \mathcal{N}(x):=\sharp \{[y]_\sim \in \Cf^{J_n}/\sim:|\hat{y}|=|\hat{x}|\} \,.
\end{equation*}


\begin{corollary}\label{cor:ambigdiscrete}
  Let $x\in \Cf^{J_n}$ and let $X$ denote its $Z$-transform.  Then
  $\mathcal{N}(x)\le 2^{L-1}$, where $L$ denotes the number of
  nontrivial factors of $X$.
\end{corollary}

In the one-dimensional case $d=1$ the $Z$-transform $X$ is a
polynomial of one variable of order $k\le n$.  By the fundamental
theorem of algebra, $X$ has $k$ roots and can be expressed as a
product of $k$ linear factors.  Assuming none of the roots lie on the
unit circle and coincides each element in the power set of the set of
roots (except for the empty set and the full set) induces a
nontrivial ambiguity.  The situation in the higher dimensional case
is radically different, as shown by Hayes and
McClellan~\cite{hayes1982reducible}.

\begin{theorem}[\cite{hayes1982reducible}]
  \label{thmhayesreducible}
  Let $\mathcal{P}^{d,k}$ denote the set of complex polynomials of
  $d>1$ variables with order $k$ and let $m$ denote the degrees of
  freedom of $\mathcal{P}^{d,k}$.  We identify $\mathcal{P}^{d,k}$
  with $\Cf^m\simeq\R^{2m}$.  Then the set of reducible polynomials in
  $\mathcal{P}^{d,k}$ is a set of measure zero (as a subset of $\Cf^m$).
\end{theorem}

Corollary \ref{cor:ambigdiscrete} together with
Theorem~\ref{thmhayesreducible} yields the following result.

\begin{corollary}\label{cor:gap1dmored}
  If $d=1$, then for any fixed $n\in \N$ the set
  $\{x\in \Cf^n: \mathcal{N}(x)<2^{n-1}\}$ is of measure zero.
  
  If $d>1$, then for any fixed $n\in \N^d$ the set
  $\{x\in \Cf^{J_n}: \mathcal{N}(x)>1\}$ is of measure zero.
\end{corollary}

A frequently used prior restriction on the signals is to require
sparsity.  To assume sparsity of the underlying signal appears natural
in many practical applications such as crystallography or astronomy.
For a thorough discussion of sparse phase retrieval among other topics,
we refer the reader to the excellent survey articles
\cite{bendory17fourier,jaganathan17phase}.  We choose to present at
this point one particular result on sparse Fourier phase retrieval
which nicely complements the univariate statement in Corollary
\ref{cor:gap1dmored}.
\begin{theorem}[\cite{oymak17sparse}]
 For $3\le k \le n-1$ let $\mathcal{S}_k^n$ denote the set of $k$-sparse signals in $\Cf^n$, i.e., the set of vectors that possess at most 
 $k$ nonzero entries, with aperiodic support.
 Then almost all $x\in\mathcal{S}_k^n$ are uniquely determined by $|\hat{x}|$ up to a constant sign factor within $\mathcal{S}_k^n$.
\end{theorem}

\subsection{Fourier phase retrieval using masks}
In the one-dimensional case the modulus of the DTFT is not a useful
representation for most signals.  A popular strategy in order to
increase information and introduce redundancy to counter the loss of
phase is to allow for masked Fourier measurements.  By a mask we mean
a function $m\in \Cf^N$ with the corresponding phaseless measurement
process being described by
$$
\Cf^N\ni x \mapsto \left|(x\odot m)\Fhat \right|,
$$
where $\odot$ denotes the pointwise product in $\Cf^N$.  In order to
attain a sufficient amount of information it is common to employ not
one but several different masks.  We essentially distinguish between
two types of this kind of measurement.  First, in case of the
short-time Fourier (STFT) measurements, the various masks
are generated by applying shifts to a fixed window function.  This is
the mathematical model behind ptychography, where an aperture is
slid over the sample to illuminate different parts (see
Figure~\ref{fig:STFT}).
Second, the various masks can be chosen in a completely unstructured manner.

We shall see that the uniqueness issues which have been discussed in
the first part of this section can be removed if the masks are
suitably chosen.  In the multivariate setting a generic signal is
uniquely (up to trivial ambiguities) determined by the modulus of its
DTFT; cf.~Corollary \ref{cor:gap1dmored}.  However, there are
deterministic signals---namely, those which possess a reducible
Z-transform---for which uniqueness fails to hold.  In a randomized
setting where one can observe the modulus of the DTFT of $x\odot m$
and the entries of the mask $m$ are drawn randomly according to a
suitable distribution, uniqueness holds with probability one provided
that the support of the signal $x$ satisfies a rather weak assumption,
as shown by Fannjiang \cite{fannjiang12absolute}.

\begin{figure}
\centering
\includegraphics[width= 0.7\textwidth]{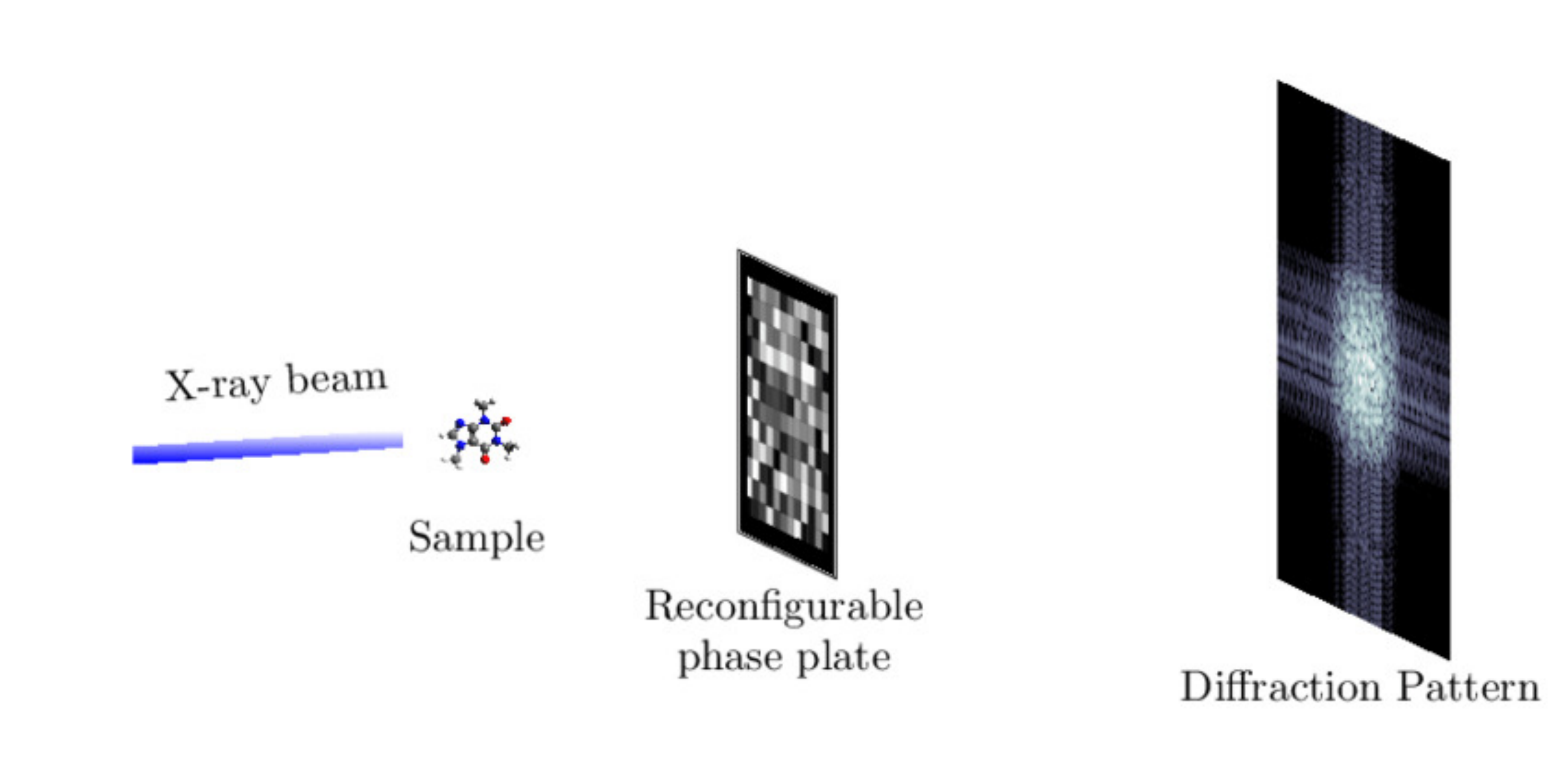}
\caption{A schematic setup of a ptychographic experiment (image
  reprinted from \cite{candes15phase} with permission from
  Elsevier). Depicted is a commonly used approach in diffraction
  imaging that can be perfectly mimicked by masked Fourier
  measurements.  The basic idea of ptychography is to create various
  diffraction patterns by illuminating different patches of the object
  one after another, see for example \cite{daSilva:15}.  To acquire
  coded diffraction patterns a mask is placed right behind the sample.
  The waveform created by the object interacts with the mask which
  results in a coded diffraction pattern.  By using different masks
  redundancy is introduced.}
\label{fig:STFT}
\end{figure}

\subsubsection{Discrete Short-Time Fourier Phase Retrieval}
\label{sec:discrete-stft-phase}

In this section, we consider finite signals $x$ in the complex Hilbert
space $\Cf^N$ with inner product
\begin{equation*}
  \langle x,y \rangle := \sum_{n=0}^{N-1} x_n \bar{y}_n \,.
\end{equation*}

The \emph{discrete Fourier transform} maps finite signals to finite signals
and is defined as
\begin{equation*}
  \hat{x}(j):= \sum_{n =0}^{N-1} x_n e^{-2 \pi i n \cdot j /N }
  \qquad \forall j \in \Z_N \,.
\end{equation*}
Its inverse is given by
\begin{equation*}
  \check{x}(j):= \frac{1}{N}\sum_{n =0}^{N-1} x_n e^{2 \pi i n \cdot j /N } 
  \qquad \forall j \in \Z_N 
\end{equation*}
and with the normalization above, Plancherel's theorem is of the form
\begin{equation*}
  \label{eq:discretePancherel}
  \langle \hat{x}, \hat{y} \rangle  = N \langle x, y \rangle \,.
\end{equation*}

We define the (circular) \emph{translation} and \emph{modulation
  operators} by
\begin{equation*}
  (T_kx)_j:= x_{j-k \pmod{N}} \qquad \text{and} \qquad (M_lx)_j:= e^{2 \pi i j \cdot l / N} x_j
\end{equation*}
for $k,l \in \Z_N$.  In the following, we identify the finite
signal $x \in \Cf^N$ with its periodic extension and just write
$(T_kx)_j=x_{j-k}$ for the circular translated signal.


Since a modulation in time corresponds to a shift in frequency,
operators of the form $\pi(\lambda)=\pi(k,l):= M_lT_k$ are called
\emph{time-frequency shifts} for $\lambda = (k,l)$.  Note
that time-frequency shifts do not commute, but satisfy the following
commutation relation.

\begin{lemma}
  Let $\lambda=(k,l), \mu = (p,q) \in \Z_N^2$.  Then
  \begin{equation}
    \label{eq:TFcommutation}
    \pi(\lambda)\pi(\mu) = e^{2 \pi i (-k \cdot q + l \cdot p)/N}
    \pi(\mu) \pi(\lambda) = e^{2 \pi i \mu \cdot \I \lambda/N}
    \pi(\mu) \pi(\lambda)
  \end{equation}
  where $\I = \left(\begin{smallmatrix}
        0 &  1 \\
        -1 & 0
      \end{smallmatrix}\right)$
    denotes the standard symplectic matrix.
\end{lemma}

We omit the proof, as it  is a straightforward verification.

The \emph{discrete short-time Fourier transform} of $x \in \Cf^N$ with
respect to the \emph{window} $g \in \Cf^N$ is defined by
\begin{equation*}
  V_gx(\lambda):= \langle x, \pi(\lambda) g \rangle   = (x \cdot T_k\bar{g})\Fhat(l)
  = \sum_{n=0}^{N-1} x_n \bar{g}_{n-k} e^{-2 \pi i n \cdot l /N}  
\end{equation*}
for $ \lambda = (k,l) \in \Z_N^2$.  

For fixed window $g$, the \stft\ $V_g$ is a linear operator that maps
finite signals in $\Cf^N$ to finite signals in $\Cf^{N\times N}$.  Due
to the linearity, we again have the trivial ambiguity
$|V_g(cx)| = |V_gx|$ for phase factors $|c|=1$.  Now the question is
whether these are the only ambiguities, and how can the original
signal be recovered.

\begin{problem}[discrete short-time Fourier phase retrieval]
  \label{prob:Dstft}
  Suppose $x \in \Cf^N$. Recover $x$ from $|V_g x|$ up to a global
  phase factor when $g \in \Cf^N$ is known.
\end{problem}

Whether Problem~\ref{prob:Dstft} has a solution depends on the choice
of the window $g$.  A sufficient condition is that the \stft\ $V_gg$
does not vanish anywhere on $\Z_N^2$.  In the following we aim at
proving this fact.

The main insight for \stft\ phase retrieval comes from the following
formula which also appears in \cite{ghobber11uncertainty} and will be
proved in what follows.
\begin{proposition}
  \label{sec:discrete-stft-phase-MagicFormula}
  Let $x, y, g, h \in \Cf^N$.  Then 
  \begin{equation}\label{eq:Dmagic}
    (V_gx \cdot \overline{V_hy})\Fhat(\lambda) 
    = N (V_yx \cdot \overline{V_hg})(-\I\lambda) \qquad \forall \lambda \in \Z_N^2 \, ,
  \end{equation}
  where $\I = \left(\begin{smallmatrix}
      0 &  1 \\
      -1 & 0
    \end{smallmatrix}\right)$
  denotes the standard symplectic matrix.
\end{proposition}
The proof of formula~\eqref{eq:Dmagic} is elementary and requires only two
things: the covariance property, which is an easy consequence of the
commutation relations~\eqref{eq:TFcommutation}, and a version of
Plancherel's theorem for the \stft.









\begin{lemma}[Covariance Property]
  \label{sec:discrete-stft-phase-discreteCov}
  Let $\lambda,\mu \in \Z_N^2$.  Then
  \begin{equation*}
    \label{eq:discreteCovariance}
    V_{\pi(\lambda)g}(\pi(\lambda)x)(\mu) = e^{2 \pi i \mu \cdot \I \lambda/N} V_gx(\mu) \,.
  \end{equation*}
\end{lemma}

\begin{proof}
  Note that time-frequency shifts are unitary operators on $\Cf^N$. Hence
  \begin{align*}
    V_{\pi(\lambda)g}(\pi(\lambda)x)(\mu) 
    &= \langle \pi(\lambda) x , \pi(\mu)\pi(\lambda) g \rangle \\
    &= e^{2 \pi i \mu \cdot \I \lambda/N}\langle \pi(\lambda) x , \pi(\lambda)
      \pi(\mu) g \rangle \\
    &= e^{2 \pi i \mu \cdot \I \lambda/N}\langle x ,\pi(\mu) g \rangle \\
    &= e^{2 \pi i \mu \cdot \I \lambda/N} V_gx(\mu) \, , 
  \end{align*}
  where we used the commutation relation \eqref{eq:TFcommutation} on
  the second line.
\end{proof}



\begin{proposition}[Orthogonality Relations]
  Let $g, h, x, y \in \Cf^N$.  Then
  \begin{equation}
    \label{eq:discreteOrthogonality}
    \langle V_gx, V_hy \rangle = N \langle x, y \rangle \langle h, g \rangle \,.
  \end{equation}
\end{proposition}

\begin{proof}
  We write the \stft\ as
  $V_gx(k,l) = (x \cdot T_k\bar{g})\Fhat(l)$ and use Plancherel's
  theorem in the sum over $l \in \Z_N$:
    \begin{align*}
     \langle V_gx, V_hy \rangle 
    &= \sum_{k,l =0}^{N-1} V_gx(k,l) \overline{V_hy(k,l)} 
      = \sum_{k=0}^{N-1} \sum_{l=0}^{N-1} (x \cdot T_k\bar{g})\Fhat(l)
      \overline{(y \cdot T_k\bar{h})\Fhat(l)} \\
    &= N \sum_{k=0}^{N-1} \sum_{n=0}^{N-1} x_n \bar{g}_{n-k} \bar{y}_n h_{n-k}
     = N \langle x, y \rangle \langle h, g \rangle \,.
  \end{align*}
\end{proof}

\begin{proof}[Proof of Proposition \ref{sec:discrete-stft-phase-MagicFormula}]
  First note that $\I^2= -I$, where $I$ denotes the identity matrix.
  Consequently,
  \begin{equation*}
    e^{-2\pi i \mu \cdot \lambda /N} V_gx(\mu) = e^{2\pi i \mu \cdot \I^2\lambda /N} V_gx(\mu)
    = V_{\pi(\I \lambda) g}(\pi(\I \lambda)x)(\mu)
  \end{equation*}
  by Lemma~\ref{sec:discrete-stft-phase-discreteCov}.  Hence, we obtain
  \begin{align*}
    (V_gx \cdot \overline{V_hy})\Fhat(\lambda) 
    &= \sum_{\mu \in \Z_N^2} V_gx(\mu) \overline{V_hy(\mu)} e^{-2\pi i \mu \cdot \lambda /N}
    = \sum_{\mu \in \Z_N^2} V_{\pi(\I \lambda) g}(\pi(\I \lambda)x)(\mu) \overline{V_hy(\mu)} \\
    &= \langle V_{\pi(\I \lambda) g}(\pi(\I \lambda)x), V_hy \rangle
    = N \langle \pi(\I \lambda)x, y \rangle \langle h, \pi(\I \lambda) g \rangle \, ,
  \end{align*}
  where we used the orthogonality
  relations~\eqref{eq:discreteOrthogonality} in the last step.

  Note that $\pi(\lambda)^* = c \pi(-\lambda)$ for a suitable phase
  factor $|c|=1$.  But these phase factors cancel when we bring both
  \tf\ shifts to the other side, hence
  \begin{equation*}
    (V_gx \cdot \overline{V_hy})\Fhat(\lambda)
    = N \langle x, \pi(-\I \lambda) y \rangle \overline{\langle g, \pi(-\I \lambda) h \rangle}
    = N (V_yx \cdot \overline{V_hg})(-\I\lambda) \,.
  \end{equation*}
\end{proof}

We can now prove a sufficient condition on the window to allow phase
retrieval.

\begin{theorem}
  \label{sec:discrete-stft-phase-InjFin}
  Let $g \in \Cf^N$ be a window with $V_gg(\lambda) \neq 0$ for all
  $\lambda \in \Z_N^2$.  Then any $x \in \Cf^N$ can be recovered from
  $|V_gx|$ up to a global phase factor.
\end{theorem}

\begin{proof}
  By Proposition~\ref{sec:discrete-stft-phase-MagicFormula} we have
  \begin{equation*}
    (|V_gx|^2)\Fhat(k,l) = N V_xx(-l,k) \cdot \overline{V_gg(-l,k)} 
    \qquad \forall k,l \in \Z_N \,.
  \end{equation*}
  If $V_gg$ has no zeros, we can recover $V_xx$.  Now we apply the
  inverse discrete Fourier transform to
  $V_xx(k,l) = (x \cdot T_k\bar{x})\Fhat(l)$ and obtain
  \begin{equation*}
    x_j \cdot \bar{x}_{j-k} =  \frac{1}{N}
    \sum_{l =0}^{N-1} V_xx(k,l) e^{2 \pi i l \cdot j /N } \,.
  \end{equation*}
  Setting $k=j$ yields
  \begin{equation*}
    x_j \cdot \bar{x}_0 = \frac{1}{N}
    \sum_{l =0}^{N-1} V_xx(j,l) e^{2 \pi i l \cdot j /N }
  \end{equation*}
  and we recover the signal $x$ up to a global phase factor after
  dividing by $|x_0|$.
\end{proof}

Theorem~\ref{sec:discrete-stft-phase-InjFin} also appears in
\cite{MR3500231}, where it is proved with the methods introduced in
\cite{MR3202304}.  Moreover, the authors also give examples
and counterexamples of window functions $g$ satisfying
$V_gg(\lambda) \neq 0$ for all $\lambda \in \Z_N^2$.

Different variants of Problem \ref{prob:Dstft} have been studied over
the years.  One possible, alternative point of view is to restrict the
problem to either sparse signals or signals that do not vanish at all
in favor of weaker assumptions imposed on the window.  We showcase one
illustrative result in this direction and point the reader toward the
articles by Jaganathan, Eldar, and Hassibi~\cite{jaganathan16stft}
and Bendory, Beinert, and Eldar~\cite{bendory17fourier} which give an
excellent overview on both uniqueness and algorithmic aspects of \stft\
phase retrieval.

\begin{theorem}[\cite{Eldar2015SparsePR}]
  Let $g\in\mathbb{C}^n$ be a window of length $W\ge 2$, where the
  length of $g$ is defined as the length of the smallest interval in
  $\mathbb{Z}_n$ containing the support of $g$.  Then every
  $x\in\mathbb{C}^n$ with nonvanishing entries is defined uniquely by
  $\left|V_g x\right|$ provided
 \begin{enumerate}[(i)]
  \item the discrete Fourier transform of $v$ defined as
  $v_n:=|g_n|^2, \, n\in \mathbb{Z}_n$
  is nonvanishing;
  \item $n\ge 2W-1$; and
  \item $n$ and $W-1$ are coprime.
 \end{enumerate}
\end{theorem}

Next, we study the case of a randomly picked window.
\begin{theorem}\label{thm:stftrandomwindow}
 There exists a set $E\subset \Cf^n$ of measure zero such that for all $g\in \Cf^n\setminus E$ the family 
 $(\pi(\lambda)g)_{\lambda\in\mathbb{Z}^2}$ allows for phase retrieval, i.e., the mapping 
 $$
 x\sim e^{i\theta}x \mapsto \left(\left|V_gx(\lambda)\right|^2\right)_{\lambda\in \mathbb{Z}_n^2}
 $$
 is injective.
\end{theorem}
\begin{proof}
To prove the theorem we closely follow the proof techniques used by 
Bojarovska and Flinth \cite[Proposition 2.1]{MR3500231} where a similar statement is shown.\\
By Theorem \ref{sec:discrete-stft-phase-InjFin} and since there are only finitely many $\lambda$, 
it suffices to show that there exists a set $E$ of measure zero such that for arbitrary but fixed $\lambda_0$ it holds that 
$$V_g g(\lambda_0)=\langle g,\pi(\lambda_0)g\rangle \neq 0.$$
Since $\pi(\lambda_0)$ is unitary there exists an orthonormal basis $(q_j)_{j=1}^n$ of $\Cf^n$ and $(\alpha_j)_{j=1}^n\subset \mathbb{T}$ 
such that 
$$
\pi(\lambda_0) = \sum_{j=1}^n \alpha_j q_j q_j^*,
$$
where $q_j^*$ denotes the conjugate transpose of the row vector $q_j$.
If we expand the window $g$ with respect to the basis, i.e.,
$g=\sum_j \beta_j q_j$, we get that
\begin{equation*}
 V_g g(\lambda_0) = \langle g,\pi(\lambda_0)g\rangle = \langle \sum_j \beta_j q_j, \sum_l \alpha_l\beta_l q_l\rangle = \sum_j \overline{\alpha_j} \left|\beta_j\right|^2.
\end{equation*}
Since $E':=\{\beta\in\Cf^n: \, \sum_j \overline{\alpha_j} \left|\beta_j\right|^2=0\}$ is a manifold of codimension one in $\Cf^n\simeq \R^{2n}$ and 
$\beta\mapsto g=\sum_j \beta_j q_j$ is an isometry it follows that 
$$
E:=\{g=\sum_j \beta_j q_j, \,\beta\in E'\}
$$
is of measure zero, and indeed for all $g\in \Cf^n\setminus E$ it holds that $V_g g(\lambda_0)\neq0$.
\end{proof}

\subsubsection{Phase retrieval with equiangular frames}
This subsection is devoted to presenting the work by Balan
\etal~\cite{balan09painless}, ``Painless reconstruction from
magnitudes of frame coefficients.''  The main results reveal that the
structure of certain, carefully designed frames
can be leveraged to derive explicit reconstruction formulas for the corresponding phase retrieval problem.

To specify the properties of the frames that we shall consider we give a few definitions.
\begin{definition}
 Let $\mathcal{H}$ be a $d$-dimensional Hilbert space.
 A finite family of vectors $\left\{f_1,\ldots,f_N\right\}\subset \mathcal{H}$ is called 
 \begin{itemize}
 \item \emph{$A$-tight frame}, with 
 frame constant $A>0$ if all $x\in \mathcal{H}$ can be reconstructed from the sequence of frame coefficients 
 $\left(\langle x,f_j\rangle\right)_{j=1}^N$ according to
 $$
 x=A^{-1} \sum_{j=1}^N \langle x,f_j\rangle f_j.
 $$
 \item \emph{uniform} A-tight frame if it is an $A$-tight frame and 
 there is $b>0$ such that $\| f_j\|=b$ for all $j\in\{1,\ldots,N\}$.
 \item $2$-uniform A-tight (or equiangular) frame if it is a uniform $A$-tight frame and there exists $c>0$ such that 
 $\left|\langle f_j,f_l\rangle\right|=c$ for all $j\neq l$.
 \end{itemize}
 \end{definition}
A simple example of an equiangular frame is given by the so-called \emph{Mercedes-Benz} frame.
\begin{exmp}
 Let $\mathcal{H}=\mathbb{R}^2\simeq \mathbb{C}$.
 Then the three vectors defined by 
 $$
 f_1:= 1, \quad f_2:=e^{2\pi i /3}, \quad f_3:=e^{4\pi i/3}
 $$
 form a $2$-uniform $3/2$-tight frame.
\end{exmp}
The size $N$ of a $2$-uniform tight frame is bounded from above in terms of the dimension of the space.

\begin{theorem}[{\cite[Proposition 2.3]{balan09painless}}]
 Let $N$ denote the number of vectors in a $2$-uniform tight frame 
 on a $d$-dimensional real or complex Hilbert space $\mathcal{H}$.
 Then $N\le d(d+1)/2$ in the real case, and $N\le d^2$ in the complex case, respectively.
\end{theorem}
We shall call a $2$-uniform tight frame \emph{maximal} if it is of maximal size, i.e., if 
$N= d(d+1)/2$ in the real case, and $N=d^2$ in the complex case.

\begin{definition}
 Let $\mathcal{H}$ be a real or complex Hilbert space. 
 A family of vectors $(e_k^{(j)})_{j\in\mathbb{J}, k\in\mathbb{K}}$ 
 with $\mathbb{J}=\{1,\ldots,d\}$ and $\mathbb{K}=\{1,\ldots,m\}$
 is said to form \emph{$m$ mutually unbiased bases} if for all $j,j'\in\mathbb{J}$ and $k,k'\in\mathbb{K}$ it holds that
 \begin{equation}\label{eq:modinnerprodmub}
  \left|\left\langle e_k^{(j)} , e_{k'}^{(j')}\right\rangle\right| = \delta_{k,k'} \delta_{j,j'} +\frac1{\sqrt{d}}(1-\delta_{j,j'}).
 \end{equation}
\end{definition}
If $(e_k^{(j)})$ form $m$ mutually unbiased bases, then it follows
from the defining equality \eqref{eq:modinnerprodmub}, that for fixed
$j$ we have that $(e_k^{(j)})_{k\in\mathbb{K}}$ is an orthonormal
basis.  Moreover, the magnitudes of the inner products attain just the
three values $0,1$, and $1/\sqrt{d}$.  In that sense the notion of
mutually unbiased bases may be regarded as a relaxation of equiangular
frames, which allows for only two values.

As in the case of equiangular frames the size of mutually unbiased
bases is bounded from above in terms of the dimension.
\begin{proposition}[{\cite[Proposition 2.6]{balan09painless}}]
  There are at most $m=d+1$ mutually unbiased bases in a
  $d$-dimensional complex Hilbert space $\mathcal{H}$.
\end{proposition}

To every $x\in\mathcal{H}$ we associate the operator 
$$
\mathcal{Q}_x: y\in \mathcal{H} \mapsto \langle y,x\rangle x,
$$
which is (up to a scaling factor) the projection onto the span of $x$.
In particular, $x$ is uniquely determined up to a sign factor by $\mathcal{Q}_x$.\\
A special case of the reconstruction formula reads as follows.
\begin{theorem}[{\cite[Theorem 3.4]{balan09painless}, special case}]\label{thm:mainpainless}
  Let $\mathcal{H}$ be a $d$-dimensional complex Hilbert space and
  suppose that $F=\{f_1,\ldots,f_N\}$ satisfies one of the following
  assumptions:
 \begin{enumerate}[(i)]
  \item $F$ forms a maximal $2$-uniform $N/d$-tight frame.
  \item $F$ forms $d+1$ mutually unbiased bases.
 \end{enumerate}
 Given a vector $x\in\mathcal{H}$ with associated self-adjoint
 rank-one operator $\mathcal{Q}_x$, then
 \begin{equation}\label{eq:recformulaequiangular}
   \mathcal{Q}_x=\frac{d(d+1)}{N} \sum_{j=1}^N \left| \langle x,f_j\rangle\right|^2 \mathcal{Q}_{f_j} - \|x\|^2 I,
 \end{equation}
 where $I$ denotes the identity operator on $\mathcal{H}$.
\end{theorem}
\begin{remark}
  Note that in \cite{balan09painless} the preceding theorem is phrased
  in terms of $2$-uniform tight frames that give rise to so-called
  projective $2$-designs.  The assumptions made in the version as
  stated above are stronger; see \cite[Example 3.3]{balan09painless}.
  
  Furthermore, note that since the frame is assumed to be
  $N/d$-tight, it follows that
$$
\|x\|^2 = \langle x,x\rangle = \left\langle \frac{d}{N} \sum_{j=1}^N \langle x,f_j\rangle f_j,x\right\rangle = 
\frac{d}{N} \sum_{j=1}^N \left|\langle x,f_j\rangle\right|^2,
$$
thus, the right-hand side of \eqref{eq:recformulaequiangular} can be computed from the magnitudes of the inner products 
$\langle x,f_j\rangle$.
\end{remark}

We conclude this subsection with a concrete example of $d+1$ mutually
unbiased bases that have Gabor structure.  The construction we
consider goes back to Alltop \cite{alltop80complex}; see also
\cite{strohmer03grassmannian}.
\begin{lemma}[Alltop]\label{lem:alltop}
 Let $p\ge 5$ be prime, let $\omega=e^{2\pi i /p}$ denote the $p$-th unit root, and define 
 \begin{equation}\label{def:bj}
  b_j(t):=p^{-1/2} \omega^{t^3+jt},\quad t\in\mathbb{Z}_p, \,j=1,\ldots,p.
 \end{equation}
Then it holds that 
\begin{equation*}
 \left| \langle T_{k}b_j,b_{j'}\rangle\right|=
 \begin{cases}
  1,&\text{if}\,j=j',k=0,\\
  0,&\text{if}\,j=j',k\neq 0,\\
  p^{-1/2}, &\text{otherwise}.
 \end{cases}
\end{equation*}
\end{lemma}
With the help of Lemma \ref{lem:alltop} we establish the following.
\begin{theorem}
 Let $p\ge 5$ be prime, let $\omega=e^{2\pi i /p}$ denote the $p$-th unit root, and define
 $g(t):=p^{-1/2}\omega^{t^3},\,t\in\mathbb{Z}_p$.
 For $1\le j,k\le p$ define vectors
 $$
 e_k^{(j)}(t):=(T_k M_j g)(t),\quad t\in\mathbb{Z}_p,
 $$
 and let $(e_k^{(p+1)})_{k=1}^p$ denote the canonical basis of $\Cf^p$.
 
Then $(e_k^{(j)})_{k,j}$ form $d+1$ mutually unbiased bases, and in particular for all $x\in\Cf^p$ it holds that 
$$
\mathcal{Q}_x = \sum_{j=1}^{d+1} \sum_{k=1}^d \left|\langle x,e_k^{(j)}\rangle\right|^2 \mathcal{Q}_{e_k^{(j)}} - \|x\|^2 I.
$$
\end{theorem}
\begin{proof}
 We only need to check that \eqref{eq:modinnerprodmub} holds true. Then the reconstruction formula follows by applying 
 Theorem \ref{thm:mainpainless}.
 The case $j=j'$ follows from the preceding lemma and from the definition of $e_k^{(d+1)}$, respectively.
 Thus, it remains to show that for $j\neq j'$ the magnitude of the inner product equals $1/\sqrt{p}$.
 We distinguish the two cases (a) $\max\{j,j'\}<d+1$ and (b) $j<j'=d+1$.\\
 Since $M_j g=b_j$, where $b_j$ is defined as in \eqref{def:bj} we have in case (a) that 
\begin{equation*}
 \left|\langle e_k^{(j)},e_{k'}^{(j')}\rangle\right| = \left|\langle T_k M_j g, T_{k'} M_{j'}g\rangle\right|
			    = \left|\langle T_k b_j,T_{k'}b_{j'}\rangle\right|
			    = \left|\langle T_{k-k'} b_j,b_{j'}\rangle\right| =p^{-1/2},
\end{equation*}
where the last equality follows again from Lemma \ref{lem:alltop}.
In case (b) we have that 
\begin{equation*}
  \left|\langle e_k^{(j)},e_{k'}^{(j')}\rangle\right|= \left|(T_k M_j g)(k')\right| = \left|(M_j g)(k'-k)\right| = \left|g(k-k')\right|=p^{-1/2}.
\end{equation*}

\end{proof}

\subsubsection{Lifting Methods}
Inspired by the pioneering work of Cand\`es \etal
\cite{candes13phase,candes13} which produced the now famous
\emph{PhaseLift} algorithm, the use of methods from semidefinite
programming to solve phase retrieval problems has become hugely
popular in recent years.

Given $\Phi=(\phi_k)_{k=1}^m \subset \Cf^{n}$ let us consider the
associated phase retrieval problem, i.e., the problem of finding
$x\in \Cf^n$ from observations
\begin{equation}\label{def:phaselessmeasmtsyk}
y_k:=\left|\langle x,\phi_k \rangle\right|^2, \quad k=1,\ldots,m.
\end{equation}
PhaseLift and related methods are 
based on the idea of lifting the signal $x$ of interest to a rank (at most) 
one matrix $X$ by virtue of
$$
X:=xx^*,
$$
where $x^*$ denotes the conjugate transpose of $x$. Note that $X$ determines $x$ up to a constant phase factor.

Since 
$$
y_k= \left|\langle x,\phi_k \rangle\right|^2 = \phi_k^* x \left(\phi_k^* x\right) = \phi_k^* X \phi_k.
$$
the original phase retrieval can be reformulated in terms of an optimization problem in $X$:
\begin{equation}\label{eq:minimizationlift1}
\begin{aligned}
\min_{X} \quad & \rank(X)\\
\textrm{subject to} \quad & \phi_k^* X \phi_k = y_k, \quad k=1,\ldots,m\\
  &X \succeq 0    \\
\end{aligned}
\end{equation}
In order to obtain a feasible problem---rank minimization is NP-hard
in general---\eqref{eq:minimizationlift1} is relaxed by means of
trace minimization, which gives rise to the semidefinite program
\begin{equation}\label{eq:minimizationlift2}
\begin{aligned}
\min_{X} \quad & \trace(X)\\
\textrm{subject to} \quad & \phi_k^* X \phi_k = y_k, \quad k=1,\ldots,m\\
  &X \succeq 0    \\
\end{aligned}
\end{equation}
As it turns out \eqref{eq:minimizationlift2} and the original phase
retrieval problem are equivalent if the measurement vectors are
sufficiently many and picked at random.
\begin{theorem}[\cite{candes13}]
 Consider $x\in \Cf^n$ arbitrary. 
 Suppose that 
 \begin{itemize}
 \item the number of measurements obeys $m\ge c_0 n\log n$, where
   $c_0$ is a sufficiently large constant and
 \item the measurement vectors $\phi_k, \,k=1,\ldots,m,$ are
   independently and uniformly sampled on the unit sphere of $\Cf^n$.
 \end{itemize}
 Then with probability at least $1-3e^{-\gamma m/n}$, where $\gamma$ is a positive 
 constant, \eqref{eq:minimizationlift2} has $X=xx^*$ as its unique solution. 
\end{theorem}
Let us point out that \cite{candes13} also contains 
a result that guarantees robust reconstruction of $x$ by an adequate modification of 
\eqref{eq:minimizationlift2} from noisy measurements
$$
\tilde{y_k}:=y_k + \nu_k,
$$
where $\nu$ models the effect of noise. 
Since our main focus lies on Fourier-type measurements we refrain from going into detail.

In a followup article it was shown by Cand\`es, Li, and
Soltanolkotabi~\cite{candes15phase} that also Fourier-type
measurements can be accommodated for in the PhaseLift framework if random
masks are employed.  In the case of masked Fourier measurements, with
masks $m^{(l)}\in\Cf^n$, $l=1,\dots, L$ the quantities that can be
observed are
\begin{equation*}
y_{l,k}:=\left|\sum_j x_j m_j^{(l)} e^{-2\pi i jk/n}\right|^2
       =\left| f_k^* D_l x\right|^2 = f_k^* D_l X D_l^* f_k,
\end{equation*}
where $f_k:= (e^{2\pi i kj/n})_{j=0}^{n-1}$ and $D_l:=\diag\left(m^{(l)}_j\right)_{j=0}^{n-1}.$
To set up the corresponding semidefinite program it is enough to use  
\begin{equation*}
\phi_{l,k}:=D_l^*f_k
\end{equation*}
as replacements for $\phi_k$ in \eqref{eq:minimizationlift2}:
\begin{equation}\label{eq:minimizationlift3}
\begin{aligned}
\min_{X} \quad & \trace(X)\\
\textrm{subject to} \quad & \phi_{l,k}^* X \phi_{l,k} = y_{l,k}, \quad l=1,\ldots, L, ~k=1,\ldots,n\\
  &X \succeq 0    \\
\end{aligned}
\end{equation}
Again, if the measurements are sufficiently many and suitably
randomized the trace relaxation \eqref{eq:minimizationlift3} is exact.
\begin{theorem}[\cite{candes15phase}]\label{thm:uniqcdp}
 Let $x\in\Cf^n$ be arbitrary. Suppose that 
 \begin{itemize}
  \item the number of coded diffraction patterns $L$ obeys $L\ge c \log^4 n$ for some numerical constant $c$; and
  \item the diagonal matrices $D_l, ~l=1,\ldots,L$, are independent
    and identically distributed (i.i.d.) copies of a diagonal matrix
    $D$, whose entries are themselves i.i.d.~copies of a random
    variable $d$, where $d$ is assumed to be symmetric, $|d|\le M$ as
    well as
  $$
  \mathbb{E} d=0, \quad \mathbb{E}d^2=0,\quad \mathbb{E}|d|^4=2\left(\mathbb{E}|d|^2\right)^2.
  $$
 \end{itemize}
 Then with probability at least $1-1/n$ it holds that $X=xx^*$ is the
 only feasible point of \eqref{eq:minimizationlift3}.
 \end{theorem}
 In particular, Theorem~\ref{thm:uniqcdp} implies uniqueness with high
 probability when using on the order of $\mathcal{O}(\log^4 n)$ random
 masks, which amounts to a total number of $\mathcal{O}(n\log^4 n)$
 measurements.  In the same setting, Gross, Krahmer, and
 Kueng~\cite{gross17improved} were able to prove that in fact
 $\mathcal{O}(\log^2 n)$ masks are sufficient.

 A major drawback of PhaseLift and related methods is that with
 increasing dimension the algorithms become computationally demanding.
 After all, the lifting shifts the problem from an $n$-dimensional
 space to an $n^2$-dimensional space.  A more efficient procedure that
 does not rely on lifting the signal and still comes with recovery
 guarantees is the \emph{Wirtinger flow} as proposed by Cand\`es, Li,
 and Soltanolkotabi~\cite{candes15wirtinger}.  Wirtinger flow is based
 on carefully picking an initialization using a spectral method
 followed by an iterative scheme akin gradient descent.

\subsubsection{Polarization Methods}
Within this subsection we outline the approach taken by Alexeev
\etal~\cite{alexeev14phase} and
summarize some results based upon these ideas.

Again, the objective is to reconstruct $x$ up to a phase factor from
measurements $y_k$, as given by \eqref{def:phaselessmeasmtsyk}.  At
the heart of the proposed reconstruction method lies the elementary
\emph{Mercedes-Benz} polarization identity,
\begin{equation}\label{eq:mercbenz}
 \bar{a}b=\frac13 \sum_{k=0}^2 \zeta^k\left|a+\zeta^{-k}b\right|^2, \quad a,b\in\Cf,
\end{equation}
where $\zeta:=e^{2\pi i/3}$.
Apply \eqref{eq:mercbenz} to $a=\langle x,\phi_j\rangle $ and $b=\langle x,\phi_l\rangle$ 
to obtain 
$$
\overline{\langle x,\phi_j\rangle } \langle x,\phi_l\rangle = \frac13 \sum_{k=0}^2 \zeta^k\left|\langle x, \phi_j +\zeta^k \phi_l \rangle \right|^2;
$$
thus, in particular the relative phase 
$$
\frac{\overline{\langle x,\phi_j\rangle }}{\left|\langle x,\phi_j\rangle\right|} 
\frac{\langle x,\phi_l\rangle}{\left|\langle x,\phi_l\rangle\right|}
$$
between two measurements can be observed under the assumption that we
have access to supplementary phaseless measurements
$\left|\langle x, \phi_j +\zeta^k \phi_l \rangle \right|^2$, and
provided that $y_j$ and $y_l$ do not vanish.  Hence in that case phase
information can be propagated, meaning that assuming that, the phase of
$\langle x,\phi_j\rangle$ is known, one can also reconstruct the phase
and
thus the value of $\langle x,\phi_l\rangle$.

It turns out to be very useful to represent the measurement setup as a graph
$G=(V,E)$ with each vertex corresponding to one of the measurement
vectors $\phi_j$ and vice versa.  To the edge between vertices
$\phi_j$ and $\phi_l$ the three new measurement vectors
$$
\left(\phi_j+\zeta^k\phi_l\right)_{k=0}^2
$$
are attached.  Note that if $G$ is fully connected we end up with
$\mathcal{O}(m^2)$ measurement vectors.  The main result in
\cite{alexeev14phase} guarantees that if the vectors
$(\phi_j)_{j=1}^m$ are drawn randomly and $m$ is of order $n\log n$,
then the graph $G$ can be adaptively pruned in such a way that the
resulting graph $G'$ possesses at most $\mathcal{O}(n\log n)$ edges,
and, moreover, that $x$ can be uniquely recovered from the
corresponding $\mathcal{O}(n\log n)$ measurements.  The statement
holds true with high probability.

It is important to mention that the proposed algorithm does in general
not produce an ensemble of measurement vectors such that any $x$ can
be recovered since the process of selecting the measurement vectors
depends on the observed intensities $y_k$.  Furthermore, the algorithm
can be adapted in such a way that also reconstruction from noisy
measurements can be accommodated.

In a followup paper \cite{bandeira14phase} it was shown by some of the
authors of \cite{alexeev14phase} and collaborators that a similar
approach can be taken when dealing with masked Fourier measurements.
To be more concrete, the main results reveal that the proposed
randomized procedure yields $\mathcal{O}(\log n)$ masks such that any
signal is uniquely determined up to global phase by the corresponding
phaseless masked Fourier measurements.

A robust version of the algorithm, which is capable of handling noisy
masked Fourier measurements, was recently introduced by Pfander
and Salanevich~\cite{pfander19robust}.


\section{Infinite Dimensional Fourier Phase Retrieval}
\label{sec:infpr}

This section is devoted to phase retrieval problems where the
underlying operator is the continuous Fourier transform or variants
thereof.  Such problems are typically studied within the scope of
complex analysis.  As is widely known analytic functions of several
complex variables behave very differently from univariate holomorphic
functions.  As we shall see a qualitative gap between the
one-dimensional and the multidimensional cases is also encountered for
the problem of Fourier phase retrieval.

\subsection{The Classical Fourier Phase Retrieval}

For signals of continuous variables we will use the following
normalization of the Fourier transform.
\begin{definition}
  Let $f\in L^1(\Rd)$. The \emph{Fourier transform} $\hat{f}$ of $f$
  is defined by
  \begin{equation*}\label{def:ft}
    \F f(\xi):=\hat{f}(\xi):=\int_{\Rd} f(x) e^{-2\pi i \xi\cdot x} \dif x 
    \qquad \forall \xi \in \Rd \,.
  \end{equation*}
  For $f \in L^2(\Rd)$, the Fourier transform is to be
  understood as the usual extension.
\end{definition}

The problem of continuous Fourier phase retrieval can now be stated as follows.
\begin{problem}[Fourier phase retrieval, continuous]\label{prob:fpr}
  Suppose $f\in L^2(\Rd)$ and is compactly supported. Recover $f$ from
  $|\hat{f}|$.
\end{problem}

We start with identifying the
trivial ambiguities.  As in the discrete case, let $T_\tau$ denote the
\emph{translation operator} $T_\tau f (x) = f( x-\tau)$ for
$ \tau\in \Rd $ and $R$ the \emph{reflection operator} $Rf(x)= f(-x)$.


\begin{proposition}\label{prop:fourierambigcont}
  Let $f \in L^2(\Rd)$.  Then each of the following choices of $g$
  yields the same Fourier magnitudes as $f$, i.e.,
  $|\hat{g}|=|\hat{f}|$:
  \begin{enumerate}
  \item $g=cf$ for $|c|=1$;
  \item $g=T_\tau f$ for $\tau\in \R^d$;
  \item $g=\overline{Rf}$.
  \end{enumerate}
\end{proposition}

The proof is straightforward.  Again the ambiguities of
Proposition~\ref{prop:fourierambigcont} and their combinations are
considered trivial
ambiguities.


A standard assumption is to consider only compactly supported
functions.  In the context of imaging applications, this restriction
is rather mild as it requires the object of interest to be of finite
extent.  The great advantage of this assumption is that the Fourier
transform of compactly supported functions extends analytically to all
of $\Cf^d$ and one can draw upon complex analysis and the theory of
entire functions in particular.  By the well-known Paley--Wiener
theorem~\cite{paley1987fourier} for functions of one variable the
converse also holds true.  The extension to higher dimensions is due to
Plancherel and P{\'o}lya~\cite{plancherel1937fonctions}.

\begin{theorem}[Paley--Wiener]
  \label{thm:paleywiener}
  Let $f\in L^2(\Rd)$ be compactly supported.  Then its
  \emph{Fourier--Laplace transform},
  \begin{equation*}
    \label{eq:fourierext}
    F(z): =\int_{\Rd} f(x) e^{-2\pi i z\cdot x} \dif x \qquad \forall z\in \Cf^d,
  \end{equation*}
  is an entire function of exponential type; \ie there exist
  $C_1,C_2 >0$ such that
  \begin{equation*}
    |F(z)| \le C_1 e^{C_2 |z|} \qquad \forall z\in \Cf^d\,.
  \end{equation*}

  Conversely, suppose $F: \Cf^d \to \Cf$ is an entire function of
  exponential type and its restriction to the real plane
  $F|_{\Rd}: \Rd \to \Cf$ is square integrable.  Then $F$ is the
  Fourier--Laplace transform of a compactly supported function
  $f \in L^2(\Rd)$.
\end{theorem}
\begin{remark}
  Let us mention that Theorem \ref{thm:paleywiener} can be extended to
  compactly supported distributions.  This result is also known by the
  name of \emph{Paley--Wiener--Schwartz}. See \cite[Chapter
  7]{hormander03analysis} for more details.
\end{remark}


\begin{definition}
  An entire function $F$ of one or several variables is called
  \emph{reducible} if there exist entire functions $G,H\neq 0$ both
  having a nonempty zero set such that $F=G\cdot H$.  Otherwise $F$
  is called \emph{irreducible}.
\end{definition}

The decomposition of an entire function of exponential type into
irreducible factors is unique up to nonvanishing factors.  For
functions of one variable this is due to the Weierstrass factorization
theorem \cite{krantz1999handbook}, and for functions of several
variables, to Osgood~\cite{osgood1965lehrbuch}.  A similar result as
in the discrete case (cf.~Theorem~\ref{thm:discreterecubile}) can be
established.

\begin{theorem}\label{thm:ambigcontinuous}
  Let $f,g\in L^2(\R^d)$ be compactly supported and let $F,G$ denote
  the Fourier--Laplace transform of $f$ and $g$ respectively.  Then
  $|\hat{f}|=|\hat{g}|$ if and only if there exists a factorization
  $G=G_1\cdot G_2$, a constant $\gamma$ with $|\gamma|=1$, and
  an entire function $Q$ where $Q\big|_{\R^d}$ is real-valued such that
 \begin{equation}\label{eq:fourierlaplacefg}
  F(z)=\gamma e^{i Q(z)} \cdot G_1(z)\cdot \overline{G_2(\bar{z})} \,.
 \end{equation} 
\end{theorem}
\begin{proof}
  The proof is quite similar to the proof of Theorem
  \ref{thm:discreterecubile}.  Therefore we give only a sketch.
  First, from the assumption that $|\hat{f}|=|\hat{g}|$, it follows
  by analytic extension that
  \begin{equation}\label{eq:entfcts}
    F(z)\cdot \overline{F(\bar{z})} = G(z)\cdot \overline{G(\bar{z})} 
    \qquad \forall z\in\Cf^d\,.
  \end{equation}
  Both $F$ and $G$ can be represented as (infinite) products of
  irreducible functions, where the representations are essentially
  unique.  By plugging the product expansions into \eqref{eq:entfcts},
  one can finally deduce in a similar way as in the proof of Theorem
  \ref{thm:discreterecubile} that \eqref{eq:fourierlaplacefg}
  holds true.

  Sufficiency follows from the observation that the function defined
  by the right-hand side of \eqref{eq:fourierlaplacefg} has the same
  modulus as $G$ for arguments in $\Rd$.
\end{proof}

The constant $\gamma$ and the modulation $e^{2\pi i \tau\cdot z}$ in
formula~\eqref{eq:fourierlaplacefg} correspond to multiplication by a
unimodular constant and translation in the signal domain respectively.
Flipping the whole Fourier--Laplace transform, i.e., choosing
$G(z)=G_2(z)=\overline{F(\bar{z})}$, amounts to reflection and
conjugation of the underlying function.

By making use of the Paley--Wiener theorem, we can characterize all
ambiguous solutions of a given function $f$.
\begin{corollary}
  Let $f\in L^2(\Rd)$ be compactly supported, and let $F$ denote its
  Fourier--Laplace transform.  Furthermore, suppose that
  $F=F_1\cdot F_2$  such that the entire function
  $G(z):= F_1(z) \cdot \overline{F_2(\bar{z})}$ is of exponential
  type.  Then for any constant $\gamma$ with $|\gamma|=1$ and
  $\tau \in \Rd$ the function
  \begin{equation*}
     g:= \gamma \cdot T_\tau \mathcal{F}^{-1}(G|_{\Rd})
  \end{equation*}
  is ambiguous with respect to $f$, i.e., $|\hat{g}|=|\hat{f}|$.  
  Here $G|_{\Rd}$ denotes the restriction of $G$ to real-valued inputs
  and $\F$ is the usual Fourier transform on $\Rd$.
\end{corollary}
For functions of one variable the question of uniqueness has been
studied in the late 1950s by Akutowicz \cite{akutowicz1956determination,
  akutowicz1957phase2} and a few years later independently by Walther
\cite{walther1963question} and Hofstetter
\cite{hofstetter1964construction}.  Their results reveal that all
ambiguous solutions of the phase retrieval problem are obtained by
flipping a set of zeros of the holomorphic extension of the Fourier
transform across the real axis.

\begin{theorem}[Akutowicz--Walther--Hofstetter]
  \label{thm:zeroflip}
  Let $f,g\in L^2(\R)$ be compactly supported and let $F,G$ denote
  their respective Fourier--Laplace transforms.  Let $m$ denote the
  multiplicity of the root of $F$ at the origin, and let $Z(F)$ denote
  the multiset of the remaining zeros of $F$, where all zeros appear
  according to their multiplicity.  Then $|\hat{f}| = |\hat{g}|$ if
  and only if there exist $a,b\in \R$ and $E\subset Z(F)$ such that
  \begin{equation}
    \label{eq:zeroflip}
    G(z)= e^{i(a+bz)} z^m \cdot \prod_{\zeta\in E}(1-z/\bar{\zeta})e^{z/\bar{\zeta}}
    \cdot \prod_{\zeta \in Z(F)\setminus E} (1-z/\zeta) e^{z/\zeta}.
  \end{equation}

\end{theorem}
Theorem~\ref{thm:zeroflip} can be deduced similarly to
Theorem~\ref{thm:ambigcontinuous} by making use of Ha\-da\-mard's factorization
theorem (see, for example, \cite{ahlfors1966complex}), which states that
an entire function of one complex variable is essentially determined
by its zeros.  More precisely, suppose $F$ is an entire function of
exponential type with a zero of order $m$ at the origin.
Then there exist $a,b\in \Cf$ such that
\begin{equation}\label{eq:hadamardfactorizationF}
F(z)=e^{az+b} z^m \cdot \prod_{\zeta \in Z(F)}(1-z/\zeta)e^{z/\zeta}\,.
\end{equation}

For the sufficiency part of Theorem~\ref{thm:zeroflip} it is important
to point out that the right-hand side of \eqref{eq:zeroflip}
constitutes an entire function of exponential type for every choice of
$E\subset Z(F)$. This follows from a result due to
Titchmarsh~\cite{Titchmarsh26}.

While for functions of one variable the expectation of uniqueness is
in general hopeless, it is commonly asserted that---similar to the
finite-dimensional case (\cf theorem~\ref{thmhayesreducible})---the
situation changes drastically when switching to multivariate
functions; see \cite{barakat1984necessary}, where it is stated that
\begin{quotation}
  Irreducibility extends to general functions of two variables with
  infinite sets of zeros, so that exact alternative solutions are most
  unlikely in 2-D phase retrieval.
\end{quotation}
However, we are not aware of a rigorous argument of this claim.

\subsection{Restriction of the 1D Fourier Phase Retrieval Problem}
\label{sec:restr-modif-1d}

Common restriction approaches to achieve uniqueness include the
following: demand the function (1) to be real-valued or even positive;
(2) to satisfy certain symmetry properties; (3) to be monotonic; or
(4) to be supported in a prescribed region.  We will only state an
incomplete, deliberate selection of results in this direction.
Before that we mention that requiring positivity as the only a priori
assumption (apart from compact support) does not suffice for
$|\hat{f}|$ to uniquely determine $f$ up to trivial ambiguities, as
has been shown in \cite{Crimmins1981ambiguity}.

\begin{theorem}
  Suppose that $f\in L^2(\R)$ is compactly supported and that there
  exists $t_0\in \R$ such that
  \begin{equation*}
    \overline{f(t_0-t)} = f(t_0+t) \qquad \forall t\in \R.    
  \end{equation*}
  Then $f$ is uniquely (up to trivial ambiguities) determined by
  $|\hat{f}|$.
\end{theorem}

\begin{proof}
  As translations are trivial ambiguities, we may assume w.l.o.g.\@
  that $t_0=0$.  Due to the symmetry of $f$, its Fourier--Laplace
  transform $F$ satisfies
  \begin{equation}
    \label{eq:restr-modif-1d-symFT}
    \overline{F(\bar{z})} = \int_\R \overline{f(t)} e^{2\pi i z t} \dif t 
    = \int_\R f(-t) e^{2\pi i z t} \dif t = F(z) \qquad \forall z \in \Cf \,.
  \end{equation}
  Particularly, the zeros of $F$ appear symmetrically with respect to
  the real axis.  Uniqueness now follows from the observation that if
  any factor of the Hadamard factorization
  \eqref{eq:hadamardfactorizationF} is to be flipped, then necessarily
  also the factor corresponding to its complex conjugate must be
  flipped in order to preserve property
  \eqref{eq:restr-modif-1d-symFT}.  Thus the flipping procedure cannot
  introduce ambiguous solutions.
\end{proof}

We have seen in the previous theorem that by requiring $f$ to be
symmetric, the zeros of its Fourier--Laplace transform appear in a
symmetric way, which ensures uniqueness.  By requiring that $f$ be
monotonically nondecreasing, it can be shown that all the zeros of
the Fourier--Laplace transform are located in the lower half-plane,
which gives the following result.

\begin{theorem}[\cite{klibanov1995phase}]
  Suppose that $f$ is  supported in an interval $[a,b]$, positive, and
  monotone on $[a,b]$.  Then $f$ is uniquely (up to trivial
  ambiguities) determined by $|\hat{f}|$.
\end{theorem}

A further method to enforce uniqueness is to require the function to
be supported on two intervals which are sufficiently far apart from
each other.

\begin{theorem}[\cite{Greenaway:77,Crimmins1983uniqueness}]
  Suppose that $f=f_1+f_2\in L^2(\R)$, where the support of $f_1$ and
  $f_2$ is contained in finite, disjoint intervals $I_1$ and $I_2$
  respectively.  Suppose further that the distance between the
  intervals $I_1$ and $I_2$ is greater than the sum of their lengths
  and that the Fourier--Laplace transforms of $f_1$ and $f_2$ have no
  common zeros.  Then $f$ is uniquely (up to trivial ambiguities)
  determined by $|\hat{f}|$.
\end{theorem}

\subsection{Additional Measurements}

The use of a second measurement obtained by additive distortion by a
known signal has also been considered.

\begin{theorem}[\cite{klibanov1995phase}]
  Suppose $g\in L^2(\R)$ is compactly supported and its Fourier
  transform is real valued and suppose $f\in L^2(\R)$ with compact
  support in $[0,\infty)$.  Then $f$ is uniquely determined by
  $|\hat{f}|$ and $|\hat{f}+\hat{g}|$.
\end{theorem}

If the additive distortion $g$ is chosen to be a suitable multiple of
the delta distribution the magnitude information of $\hat{f}$ is
dispensable, i.e., if $c$ is sufficiently large compared to $f$, then
$f$ can be recovered from $|\hat{f}+c|$.  The interference of $f$ with
such a $g$ pushes all the zeros of the analytic extension of the
Fourier transform to the upper half-plane.  In this case the relation
between phase and magnitude is described by the Hilbert transform
\cite{Burge1976phase,burge1974application} and remarkably, phase
retrieval is rendered not only unique but also stable.

\begin{theorem}
  \label{thm:deltainf}
  For $a,b>0$ we define
  $\mathcal{B}_{a,b}:=\{f\in L^2(\R): \|f\|_{L^\infty(\R)}<a
  \allowbreak \text{and } \supp(f)\subseteq [0,b]\}$
  and for $c\in \R$ let $L^2_c(\R):= \{f+c: f\in L^2(\R)\}$ endowed
  with the metric
  \begin{equation*}
    d_{L^2_c(\R)}(f,g):=\|f-g\|_{L^2(\R)}.    
  \end{equation*}
  Suppose $c>ab$.  Then $\mathcal{A}:f\mapsto |\hat{f}+c|$ is an
  injective mapping from $\mathcal{B}_{a,b}$ to $L^2_c(\R)$ and
  $\mathcal{A}^{-1}: \mathcal{A}(\mathcal{B}_{a,b}) \subseteq
  L^2_c(\R) \rightarrow \mathcal{B}_{a,b}$
  is uniformly continuous, i.e. there exists a constant $C>0$ such
  that
  \begin{equation*}
    \label{est:unifcont}
    \|f_1-f_2\|_{L^2(\R)} \le C \cdot d_{L^2_c(\R)} (|\hat{f_1}+c|,|\hat{f_2}+c|)
     \qquad \forall f_1,f_2\in \mathcal{B}_{a,b}.
 \end{equation*}
\end{theorem}
\begin{proof}
  In order to show that $\mathcal{A}$ maps from $\mathcal{B}_{a,b}$ to
  $L^2_c(\R)$ let $f\in  \mathcal{B}_{a,b} \subseteq L^2(\R)$ be
  arbitrary.  We have to verify that $\mathcal{A}f-c\in L^2(\R)$.  By
  the reverse triangle inequality we have that
  \begin{equation*}
     |\mathcal{A}f-c|=||\hat{f}+c|-c|\le |\hat{f}+c-c|=|\hat{f}| \,.
  \end{equation*}
  Since $\hat{f}\in L^2(\R)$ also $\mathcal{A}f-c\in L^2(\R)$.

  Let us denote by $g$ the analytic extension of $\hat{f}+c$, i.e.,
  \begin{equation*}
    g(z)=\int_\R f(t)e^{-2\pi i z t} \dif t + c \qquad \forall z\in \Cf \,.    
  \end{equation*}

  Then---provided that $g$ has all its zeros in the upper (or lower)
  half-plane---phase and magnitude of $g$ are related via the Hilbert
  transform \cite{burge1974application}, i.e.,
  \begin{equation}\label{eq:hilbertrafo}
    \alpha(x):=H(\ln |g|)(x):= -\frac1\pi P.V. \int_\R 
    \frac{\ln |g(t)|}{t-x} \dif t \qquad \forall x\in\R \,,
  \end{equation}
  satisfies $g=|g|e^{i\alpha}$.
  
  In order to make use of this identity, we check that $g$ has no zeros
  in the lower half-plane: For $\Im z\le 0$ it holds that
  \begin{equation*}
     \Big|\int_\R f(t)e^{-2\pi i z t} \dif t\Big| \le \|f\|_{L^1(\R)} \le ab
  \end{equation*}
  and we have $ |g(z)|\ge | |\hat{f}(z)|- |c| |\ge c-ab >0 $
  in the lower half-plane since $c>ab$.

  For $f_1,f_2\in \mathcal{B}_{a,b}$ let $g_k:=\hat{f_k}+c$ and let
  $\alpha_k:=H(\ln |g_k|)$.  Then we have for $k=1,2$ that
  \begin{equation*}
    |g_k(x)|=|\hat{f_k}(x)+c| \ge c-|\hat{f_k}(x)|\ge c-ab >0 \qquad \forall x \in \R
  \end{equation*}
  and similarly that $|g_k(x)|\le c+ab.$ It follows that there exists a
  constant $C_1 > 0$ (depending on $a,b,c$) such that
  \begin{equation}
    \label{est:logabsg}
    \left| \ln|g_1(x)|-\ln|g_2(x)|\right| \le C_1 \cdot \left| |g_1(x)|-|g_2(x)| \right| 
    \qquad \forall x \in \R \,,
  \end{equation}
  which implies that the difference $\ln|g_1|-\ln|g_2|$ is an element
  of $L^2(\R)$.  According to \eqref{eq:hilbertrafo} the phase
  difference $\delta:=\alpha_1-\alpha_2$ can be computed by
  $\delta=H\left(\ln|g_1|-\ln|g_2|\right)$.  By using the well-known
  fact that the Hilbert transform is an isometry on $L^2(\R)$
  \cite{titchmarsh1937introduction} and \eqref{est:logabsg} it follows
  that there exists a constant $C_2$ (depending on $a,b,c$) such that
  \begin{equation*}
    \label{est:hilberttrafo2}
    \|\delta\|_{L^2(\R)} \le C_2 \cdot \| |g_1|-|g_2|\|_{L^2(\R)} \,.
  \end{equation*}
  Thus we obtain by using the elementary estimate
  $|1-e^{it}|\le |t|, t\in \R$, that
  \begin{equation*}
    \begin{aligned}
      \|f_1-f_2\|_{L^2(\R)} &= \|\hat{f_1}-\hat{f_2}\|_{L^2(\R)} = \|g_1-g_2\|_{L^2(\R)}\\
      &= \| |g_1|e^{i\alpha_1} - |g_2|e^{i\alpha_2}\|_{L^2(\R)}\\
      &\le \| g_1 \cdot (1-e^{-i\delta})\|_{L^2(\R)} + \| |g_1|-|g_2| \|_{L^2(\R)}\\
      &\le  \| g_1\|_{L^\infty(\R)} \cdot \|\delta\|_{L^2(\R)} + \| |g_1|-|g_2|\|_{L^2(\R)}\\
      &\le C_3\| |g_1|-|g_2|\|_{L^2(\R)},
    \end{aligned}
  \end{equation*}
  for suitable constant $C_3 >0$.
\end{proof}

\begin{remark}
  Note that the assumption $\supp f \subseteq [0,b]$ implies not only
  that $\hat{f}$ is band-limited but also $|\hat{f}|^2$ and
  $\Re \hat{f}$.  Therefore the function
  \begin{equation*}
    |\hat{f}+c|^2-c^2=|\hat{f}|^2 + 2c \Re \hat{f}    
  \end{equation*}
  is also band-limited and $|\hat{f}+c|$ can be uniquely and stably
  determined from samples.  Together with Theorem~\ref{thm:deltainf},
  this implies that any $f\in \mathcal{B}_{a,b}$ can be recovered
  stably from the samples of $|\hat{f}+c|$ on a suitable discrete set.
\end{remark}

\subsubsection{Phase Retrieval from holomorphic measurements}

By the Paley--Wiener theorem there is a one-to-one correspondence
between certain holomorphic functions and compactly supported $L^2$
functions, in the sense that the Fourier transform of such a function
extends to such a holomorphic function.  As discussed in the previous
section this observation plays a crucial
role in identifying ambiguous solutions of the classical Fourier phase retrieval problem.

There are further instances of Fourier-type transforms that produce
essentially holomorphic measurements such as the short-time Fourier
transform with Gaussian window and the Cauchy-wavelet transform, which
leads us to pose
\begin{problem}[phase retrieval from holomorphic measurements]
  Suppose $D\subset \mathbb{C}$ is open,
  $\mathcal{X}\subset \mathcal{O}(D)$ is a set of admissible functions
  and $S\subset D$.  Given $F\in\mathcal{X}$, find all
  $G\in \mathcal{X}$ such that
 $$
 \left|G(z)\right|=\left|F(z)\right| \quad \forall z\in S.
 $$
\end{problem}

If $D$ is the complex plane, $S$ is the real line, and $\mathcal{X}$
denotes the set of entire functions of exponential type whose
restriction to the real line is square integrable,
Theorem~\ref{thm:zeroflip} reveals that there is in general a huge
amount of nontrivial ambiguities, each of which is created by flipping
a set of zeros across the real axis.
 
However, if the modulus of the function is known on two suitably
picked lines (see Figure~\ref{fig:lines}), uniqueness is guaranteed.
We first consider the case with two lines passing through the origin.

\begin{figure}[t!]
    \centering
    \begin{subfigure}[t]{0.33\textwidth}
        \centering
        \includegraphics[width=\textwidth]{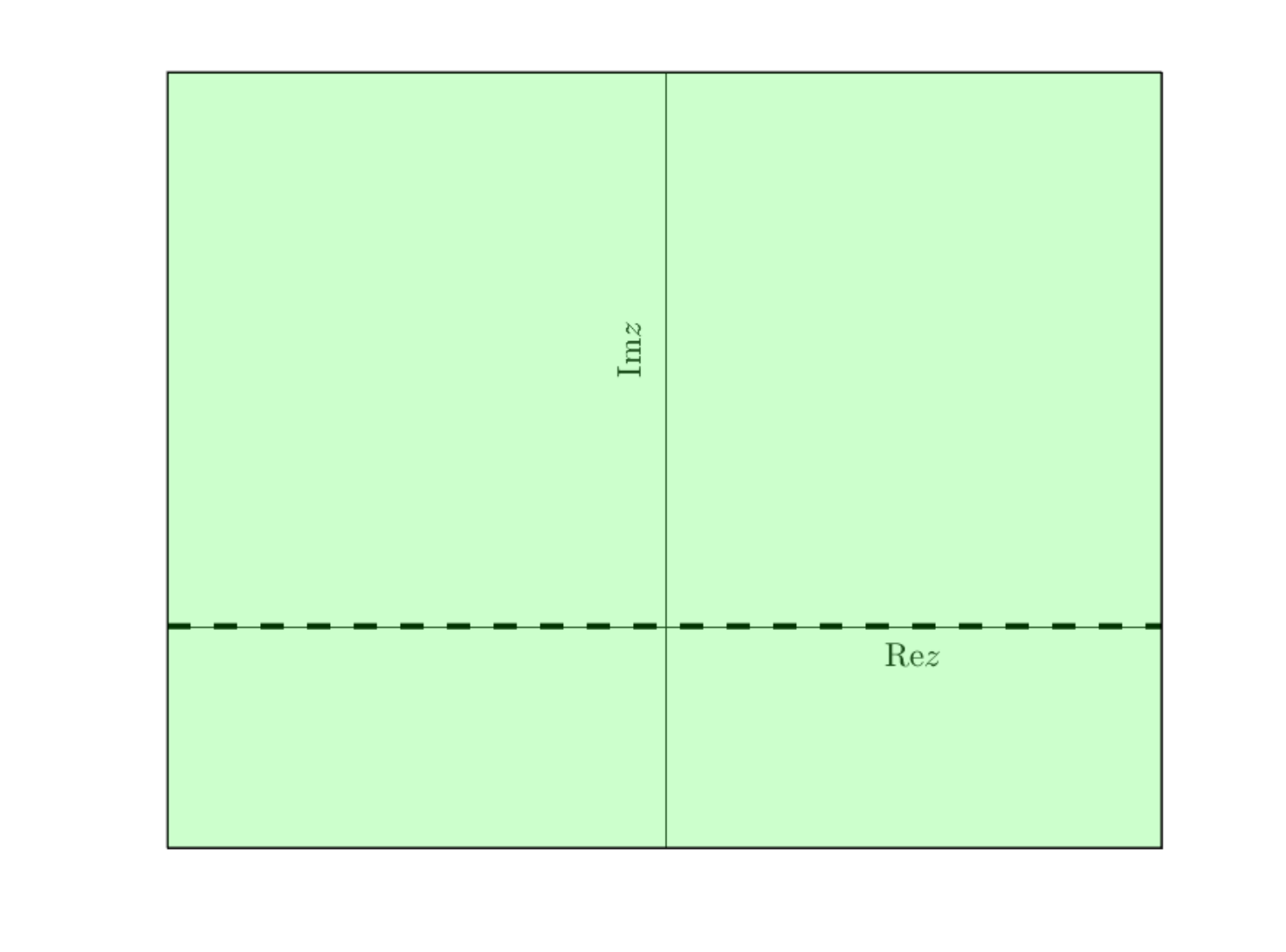}
    \end{subfigure}%
    \begin{subfigure}[t]{0.33\textwidth}
        \centering
        \includegraphics[width=\textwidth]{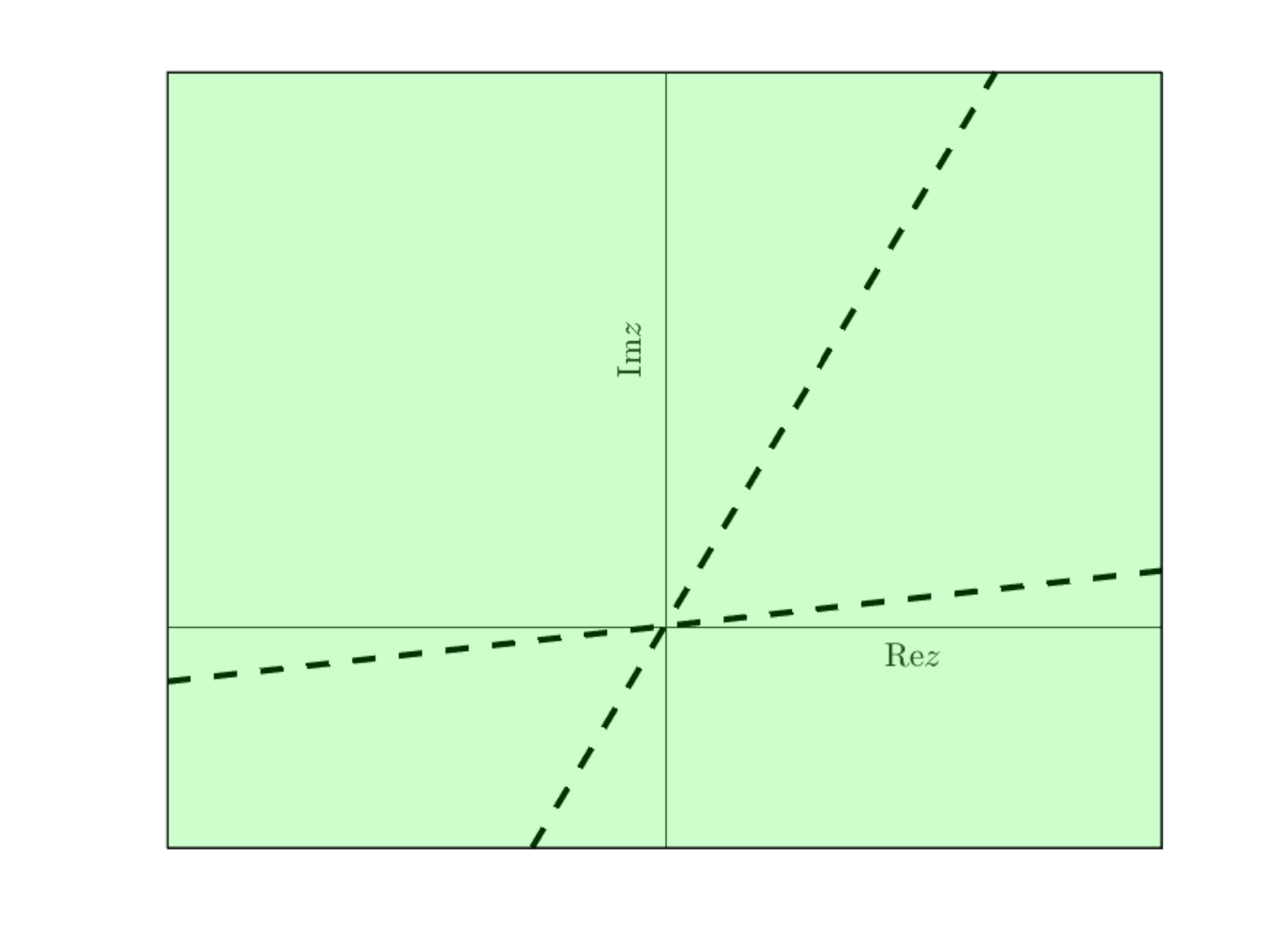}
    \end{subfigure}
    \begin{subfigure}[t]{0.33\textwidth}
        \centering
        \includegraphics[width=\textwidth]{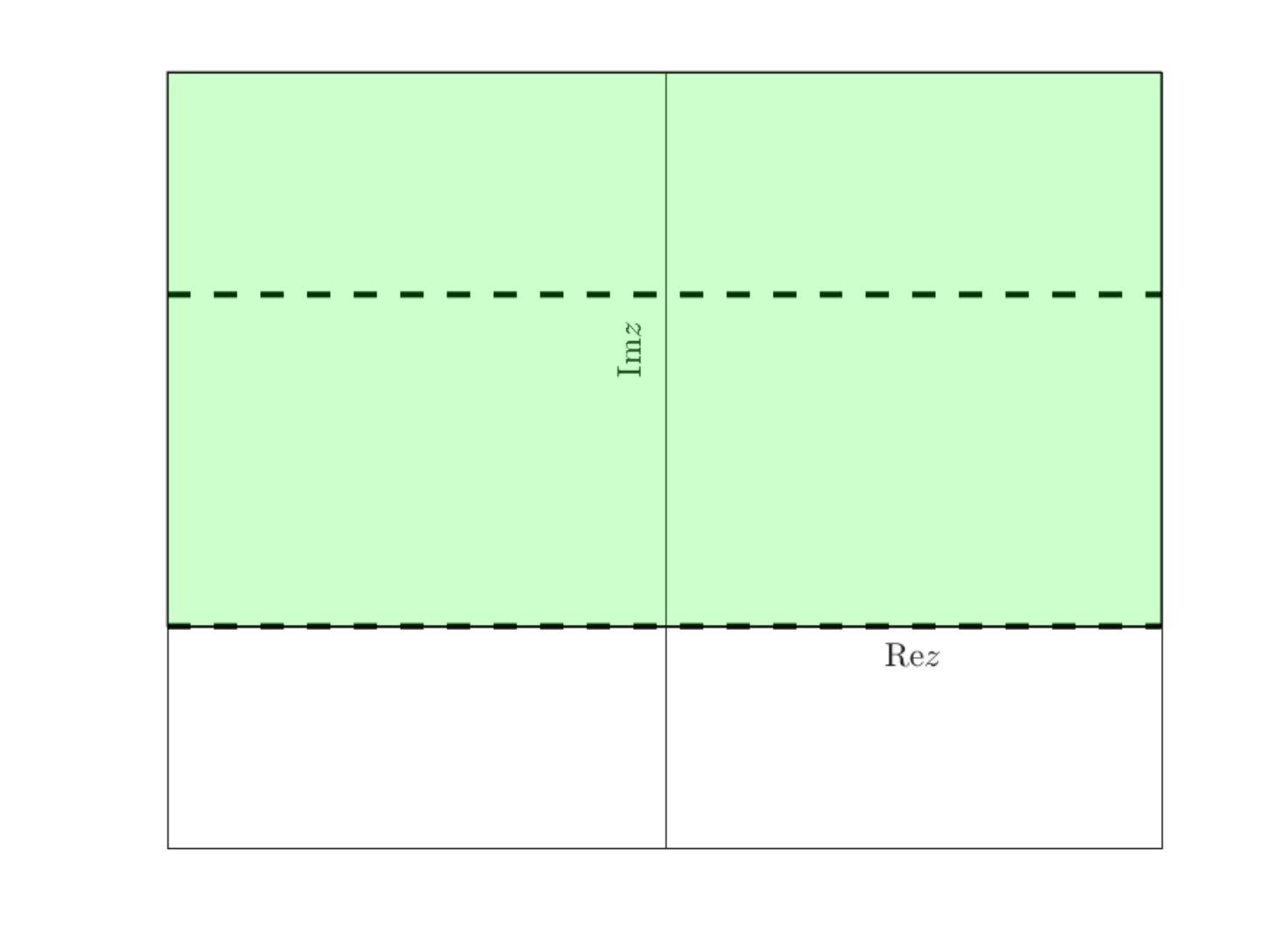}
    \end{subfigure}
    \caption{The scheme depicts the three different types of phase
      retrieval problems from holomorphic measurements that are
      considered.  The dashed lines indicate where the modulus of the
      underlying function is known; the shaded regions indicate where
      the functions to be recovered are analytic.  Reconstruction of
      an entire function given its modulus on the real line (left)
      possesses in general many nontrivial ambiguities; cf.\ Theorem
      \ref{thm:zeroflip}.  Under certain assumptions the modulus on
      two lines (center and right) uniquely determines a holomorphic
      function; cf.\ Theorems~\ref{thm:uniqholo1} and
      \ref{thm:uniqholo2}.}
    \label{fig:lines}
\end{figure}

\begin{theorem}[{\cite[Theorem 3.3]{jaming14uniqueness}}]\label{thm:uniqholo1}
  Let $\mathcal{X}$ denote the set of entire functions of finite order
  and $S$ the union of two lines passing through the origin
  $$
  S= \left\{z=te^{i\alpha_1}:\,t\in\R\right\} \cup
  \left\{z=te^{i\alpha_2}:\,t\in\R\right\},
  $$
  where $\alpha_1,\alpha_2\in [0,2\pi)$ satisfy $\alpha_1-\alpha_2\notin \pi \mathbb{Q}$.
  
  Suppose that $F,G\in \mathcal{X}$ satisfy that
  $\left|G(z)\right|=\left|F(z)\right|$ for all $z\in S$. Then there
  exists $\theta\in\mathbb{R}$ such that $G=e^{i\theta}F$.
\end{theorem}
Similarly to Theorem \ref{thm:zeroflip} the proof of Theorem
\ref{thm:uniqholo1} relies on the idea of comparing two entire
functions by making use of Hadamard's factorization theorem.  To
highlight where the assumption on the angle between the two lines
comes into play we give a sketch of the proof.
\begin{proof}[Proof sketch]
  We assume for simplicity that $F$ and $G$ are functions of
  exponential type with simple zeros and that $\alpha_1=0$.
  W.l.o.g. we may assume that $F$ and $G$ do not vanish at the origin.
  
  Let the Weierstrass factors be denoted by
 $$
 E(z,\zeta):=\left(1-z/\zeta\right) e^{z/\zeta},
 $$
 and let $Z(F)$ and $Z(G)$ denote the set of zeros of $F$ and $G$
 respectively.
 By Hadamard's factorization theorem we have that
 \begin{align*}
  F(z)&=e^{a+bz} \prod_{\zeta \in Z(F)\cap Z(G)} E(z,\zeta) \prod_{\zeta\in Z(F)\setminus Z(G)} E(z,\zeta),\\
  G(z)&=e^{c+dz} \prod_{\zeta \in Z(G)\cap Z(F)} E(z,\zeta) \prod_{\zeta\in Z(G)\setminus Z(F)} E(z,\zeta),
 \end{align*}
 with $a,b,c,d\in \mathbb{C}$.
 Since $|F|$ and $|G|$ coincide on the real line it follows that
 \begin{multline}\label{eq:comp1}
   e^{2(\Re a + z\Re b)} \prod_{\zeta\in Z(F)\setminus Z(G)} E(z,\zeta) E(z,\bar{\zeta}) \\
   = e^{2(\Re c + z\Re d)} \prod_{\zeta\in Z(G)\setminus Z(F)}
   E(z,\zeta) E(z,\bar{\zeta})
 \end{multline}
and, since $|F|$ and $|G|$ agree on the line $z=e^{i\alpha_2}t, t\in \mathbb{R}$, that 
 \begin{multline}\label{eq:comp2}
   e^{2(\Re a + z \Re be^{i\alpha_2})} \prod_{\zeta\in Z(F)\setminus
     Z(G)} E(z,e^{-i\alpha_2}\zeta)
   E(z,\overline{e^{-i\alpha_2}\zeta})\\= e^{2(\Re c + z \Re
     de^{i\alpha_2})} \prod_{\zeta\in Z(G)\setminus Z(F)}
   E(z,e^{-i\alpha_2}\zeta) E(z,\overline{e^{-i\alpha_2}\zeta})
 \end{multline}
 From \eqref{eq:comp1} and \eqref{eq:comp2} it follows that $\Re a= \Re c$ and that $b=d$. 
 Let us define the discrete set $D$ by 
  $$
  D:=(Z(F)\setminus Z(G)) \cup (Z(G)\setminus Z(F)).
 $$
 It remains to show that $D$ is the empty set.  Note that the
 identities \eqref{eq:comp1} and \eqref{eq:comp2} imply that $D$ is
 invariant under the mappings $\zeta \mapsto \bar{\zeta}$ and
 $\zeta \mapsto e^{i\alpha_2} \cdot
 \overline{e^{-i\alpha_2}\zeta}=e^{2i\alpha_2}\bar{\zeta}$,
 and thus also under their composition, which happens to be a rotation
 $$\rho:\zeta\mapsto e^{2i\alpha_2}\zeta.$$
 Assume that $D\neq \emptyset$. Then there exists $0\neq \zeta_0\in D$. Since $D$ is invariant under $\rho$ we have that the orbit 
 $$
 \omega:=\left\{\rho^n (\zeta_0): n\in \mathbb{N}\right\} =
 \left\{e^{2i\alpha_2 n}\zeta_0: n\in\mathbb{N}\right\} \subset D.
 $$
 By the assumption on $\alpha_2$ the set $\omega$ cannot be discrete---a contradiction. 
\end{proof}

For functions in the Hardy space of the upper half-plane, knowledge of
the modulus of the function on two parallel lines is sufficient.
\begin{theorem}[{\cite[Theorem 2.1]{mallat15cauchy}}]\label{thm:uniqholo2}
Let $a>0$ be fixed and
  $$
  \mathcal{X}:= \left\{F\in \mathcal{O}(\mathbb{H}): \, \sup_{y>0}
    \int_{\mathbb{R}} \left|F(x+iy)\right|^2\,dx <+\infty\right\}.
  $$
Suppose that $F,G\in \mathcal{X}$ satisfy that
\begin{enumerate}[(i)]
\item $\left|G(x+ia)\right|= \left|F(x+ia)\right|$ for almost all
  $x\in \mathbb{R}$ and
\item
  $\lim_{y\searrow0} \left|G(x+iy)\right| = \lim_{y\searrow0}
  \left|G(x+iy)\right|$ for almost all $x\in \mathbb{R}$.
\end{enumerate}
Then there exists $\theta\in\mathbb{R}$ such that $G=e^{i\theta}F$.
\end{theorem}
Since the functions considered in Theorem~\ref{thm:uniqholo2} are not
entire but only holomorphic on the half-plane, Hadamard factorization
cannot be applied in this case.  There is, however, a substitute
available, that is, functions in the Hardy space have a unique
representation as a product of its Blaschke factors, which involves
so-called inner and outer functions.
 
In \cite{mallat15cauchy}, Theorem \ref{thm:uniqholo2} is used in order
to establish uniqueness for the phase retrieval problem associated to
the Cauchy wavelet transform.  Recall that the Cauchy wavelets of
order $p>0$ are defined by
\begin{align}
  \hat{\psi}(\omega)&= \omega^p e^{-\omega}\chi_{(0,+\infty)}(\omega)\\
  \hat{\psi_j}(\omega)&=\hat{\psi}(a^j\omega), \quad j\in \mathbb{Z},\label{def:cauchywavelets}
\end{align}
where $a>1$ denotes a fixed dilation factor.
The associated wavelet transform is then given by the operator
$$
L^2(\mathbb{R})\ni f\mapsto \left( f \ast \psi_j \right)_{j\in\mathbb{Z}}.
$$
Furthermore recall that the analytic part $f_+$ of a function
$f\in L^2(\mathbb{R})$ is defined by
\begin{equation*}\label{def:anlyticpart}
 \hat{f_+}(\omega)= 2\hat{f}(\omega) \chi_{(0,+\infty)}(\omega).
\end{equation*}
\begin{theorem}[{\cite[Corollary 2.2]{mallat15cauchy}}]
  Let $(\psi_j)_{j\in\mathbb{Z}}$ be defined as in
  \eqref{def:cauchywavelets}.  Suppose $f,g \in L^2(\mathbb{R})$ are
  such that for some $j\neq k$ it holds that
 $$
 \left|g\ast\psi_j\right| = \left|f\ast\psi_j\right| \quad \text{and}
 \quad \left|g\ast\psi_k\right| = \left|f\ast\psi_k\right|.
 $$
 Then there exists $\theta\in \mathbb{R}$ such that the analytic parts
 of $f$ and $g$ satisfy
 $$ g_+=e^{i\theta} f_+.$$
\end{theorem}
\begin{remark}
  The article \cite{mallat15cauchy} also studies stability properties
  of the phase retrieval problem for Cauchy wavelets.  The authors
  observe in numerical experiments that instabilities are of a certain
  ``generic'' type and give formal arguments that there cannot be other
  types of instabilities; cf.~\cite[introduction of
  Sec. 5]{mallat15cauchy}
\begin{quotation}
  The goal of this section is to give a partial formal justification
  to the fact that has been nonrigorously discussed \dots: when two
  functions $g_1,g_2$ satisfy
  $\left|g_1\ast \psi_j\right| = \left|g_2\ast \psi_j\right|$ for all
  $j$, then the wavelet transforms $\left\{g_1\ast\psi_j(t)\right\}_j$
  and $\left\{g_2\ast\psi_j(t)\right\}_j$ are equal up to a phase
  whose variation is slow in $t$ and $j$, except eventually at the
  points where $\left|g_1\ast \psi_j(t)\right|$ is small.
\end{quotation}
\end{remark}

\subsubsection{The Pauli Problem}

In 1933 Pauli asked his seminal work \textit{Die allgemeinen Prinzipien der
Wellenmechanik}~\cite{PauliWolfgang1990DaPd} whether a wave
function is uniquely determined by the probability densities of
position and momentum.  In mathematical terms, this is equivalent to
the following phase retrieval problem known as the \emph{Pauli
  problem}.




\begin{problem}[Pauli problem]
  \label{sec:TFM-PauliProb}
  Do $|f|$ and $|\hat{f}|$ determine $f \in L^2(\R)$ uniquely?
\end{problem}

Reichenbach~\cite{Reichenbach49} published the first counterexamples
of Barg\-mann in 1944: Any symmetric $f$ and its flipped complex
conjugated function $\overline{Rf}$ have the same modulus and absolute
Fourier measurement. We will call any pair of functions which cannot
be distinguished under the measurements of
Problem~\ref{sec:TFM-PauliProb} \emph{Pauli partners}.  If a function
does not have any Pauli partners beyond the trivial ambiguity of
multiplication by a unimodular constant, it is said to be \emph{Pauli
  unique}.

In 1978 Vogt~\cite{MR0495900} (see also Corbett and
Hurst~\cite{MR0489489}) exploited the relation
$C\F = \F C R$ to produce infinitely many Pauli partners. Recall that
$Cf:=\overline{f}$ and $Rf(x):= f(-x)$ denote the conjugation and
reflection operators, respectively.  If a function satisfies the
symmetry relation $\overline{f(-x)}= f(x) w(x)$ with $|w(x)|=1$ and
$w$ is not constant on $\{x: f(x) \neq 0 \}$, then $f$ and
$\overline{Rf}$ are again Pauli partners.

Note that both counterexamples, those of Bargmann and of Vogt, Corbett
and Hurst, respectively, are trivial ambiguities of the classical
Fourier phase retrieval problem (Problem~\ref{prob:fpr}).  (But not of
the Pauli problem, whose only trivial ambiguity is multiplication by a
unimodular constant, since translations and conjugated reflections are
picked up on in general.)


Since the Pauli problem is of particular interest in quantum
mechanics, it is often studied from a quantum mechanical perspective,
where the \emph{position} and \emph{momentum operator} play a central
role.  We will use them in the following normalization
\begin{equation}
  \label{eq:PosMom}
  Qf(x):=xf(x) \text{ and } Pf(x):= -\frac{i}{2 \pi} \frac{\dif }{\dif x} f (x)
\end{equation}
such that $\F P \F^{-1} = Q$.

Corbett and Hurst~\cite{MR0489489} proved the following theorem
characterizing Pauli uniqueness.

\begin{theorem}
  Let $Q,P$ denote the position and momentum operator as defined in
  \eqref{eq:PosMom}.  Then $f \in L^2(\R)$ is Pauli unique if and only
  if there exists a $\lambda \in \R$ and real-valued Borel-measurable
  functions $F,G$ such that $e^{iF(Q)}e^{iG(P)}f = e^{i \lambda}f$.
\end{theorem}

Note that constant functions $F,G$ amount to multiplication by a
unimodular constant, \ie the only trivial ambiguity of the Pauli
problem.

\begin{proof}
  Suppose there exists a $\lambda \in \R$ and real-valued
  Borel-measurable functions $F,G$ such that
  $e^{iF(Q)}e^{iG(P)}f = e^{i \lambda}f$.  Let
  $g:= e^{iG(P)}f = e^{-iF(Q)} e^{i \lambda}f$.  By the functional
  calculus for the position operator, the operator $e^{-iF(Q)}$
  amounts to multiplying with the function $e^{-iF(\,.\,)}$.  Hence
  \begin{equation*}
    |g(x)| = |e^{-iF(x)} e^{i \lambda}f(x)| = |f(x)| \qquad \text{for a.e. } x \in \R
  \end{equation*}
  since $F$ is real-valued.

  Due to the unitary equivalence $\F P \F^{-1} = Q$, the operator
  $e^{iG(P)}$ is the multiplication operator $e^{iG(Q)}$ on the
  Fourier domain, \ie
  $\F e^{iG(P)} \F^{-1} \phi (\xi) = e^{iG(\xi)} \phi(\xi)$.  Consequently
  \begin{equation*}
    |\hat{g}(\xi)| = |\F (e^{iG(P)}f) (\xi)| = |\F e^{iG(P)} \F^{-1} (\hat{f}) (\xi)|
    = |\hat{f}(\xi)| \qquad \text{for a.e. } \xi \in \R \,,
  \end{equation*}
  which proves the necessary direction.

  Conversely, assume that $f$ has Pauli partner $g$.  Then there
  exist real-valued Borel-measurable functions $F,G$ such that
  \begin{align}
    g(x) = e^{i F(x)} f(x) \text{ and } \hat{g}(\xi) = e^{iG(\xi)} \hat{f}(\xi)
    \qquad \text{for a.e. } x, \xi \in \R \,.
  \end{align}
  Hence
  \begin{align}
    g(x) = \F^{-1} \hat{g}(x) = \F^{-1} e^{iG(Q)} \F (f) (x) = e^{i G(P)} f(x)
  \end{align}
  and therefore $e^{i F(Q)} f = g =  e^{i G(P)} f$.
\end{proof}


\begin{corollary}
  \label{sec:pauli-problem-CorHurst}
  Let $Q,P$ denote the position and momentum operator as defined in
  \eqref{eq:PosMom}.  Suppose $A(Q,P)$ is a self-adjoint operator such
  that there exists a unitary operator $U$ with
  $Ue^{iA}U^*=e^{iF(Q)}e^{iG(P)}$.  If $A\phi=\lambda \phi$, then
  $f:= U \phi$ is Pauli nonunique with Pauli partner
  $g:=e^{iG(P)}f= e^{-iF(Q)}f$.
\end{corollary}

Corbett and Hurst~\cite{MR2209781} used this result to show that there
exists a dense set of $L^2(\R)$ that is Pauli nonunique, but includes
both trivial and nontrivial solutions of the classical Fourier phase
retrieval problem.  Furthermore, they constructed uncountably many
Pauli nonunique functions which are not trivial solutions of the
classical problem.

For this, they considered the Hamiltonian of the quantum harmonic
oscillator
\begin{equation*}
  H = \frac{1}{2m}P^2 + \frac{K}{2}Q^2 \,,
\end{equation*}
where $m, K >0$ are positive constants corresponding to the mass of
the particle and the force constant, respectively.  One can show that
the self-adjoint operator $H$ satisfies
\begin{equation*}
  e^{isQ^2/4}e^{iH}e^{-isQ^2/4} = e^{isQ^2/2}e^{-itP^2/2} \,,
\end{equation*}
where
\begin{align*}
  s= \frac{2 \beta(1- \cos  \alpha)}{ \sin \alpha },
  \quad t= \frac{\sin \alpha }{ \beta} 
  \quad \text {with} 
  \quad \alpha= \frac{K^{1/2}}{2 \pi m^{1/2}} \notin  \pi \N,
  \quad \beta= 2 \pi (Km)^{1/2} \,.
\end{align*}

The eigenfunctions of the Hamiltonian satisfying
\begin{equation*}
  H \psi_k = \alpha (k+1/2) \psi_k \qquad k = 0, 1, 2, \dots
\end{equation*}
are the Hermite functions
\begin{equation*}
  \psi_k(x) = \frac{(-1)^k \sqrt{\beta}}{\pi^{1/4}\sqrt{2^n n!}} \mkern4mu
  e^{\beta x^2/2} \Big(\frac{\dif}{\dif x}\Big)^k e^{-\beta x^2} \,.
\end{equation*}
By Corollary~\ref{sec:pauli-problem-CorHurst}, the functions
\begin{equation*}
 f_k(x):= e^{isx^2/4} \psi_k(x)
\end{equation*}
are Pauli nonunique with Pauli partner 
\begin{equation*}
  g_k(x)=  e^{-isx^2/2}f_k(x) = e^{- isx^2/4} \psi_k(x) = \overline{f_k(x)}\,.
\end{equation*}
Note that due to the symmetry of the Hermite functions, the pairs
$(f_k, g_k)$ are again trivial solutions of the classical Fourier
phase retrieval problem.  Nevertheless, this construction yields an
orthonormal basis of $L^2(\R)$ of Pauli nonunique functions.

To construct nontrivial solutions in the classical sense, Corbett and
Hurst exploited the periodicity of the eigenvalues of $f_k$.  Observe that
\begin{equation*}
  e^{isQ^2/2}e^{-itP^2/2} f_k = e^{isQ^2/4}e^{iH}(\psi_k) = e^{i\alpha (k+1/2)} f_k \,.
\end{equation*}
Therefore, by defining $\Hf _{b,c}:= \vspan \{ f_{nb+c}: n \in \N \}$
where $b \geq 3$, $0 \leq c \leq b-1$ and choosing $\alpha = 2 \pi /b$,
we obtain for every $f \in \Hf _{b,c}$
\begin{equation*}
  e^{isQ^2/4} e^{iH}e^{-isQ^2/4}f = e^{i \pi (2c+1)/b}f \,.
\end{equation*}
Hence any $f = \sum_{n=0}^N a_n f_{nb+c}$ has the Pauli partner
$g= \sum_{n=0}^N a_n \overline{f}_{nb+c}\neq \overline{f}$ as long as
at least one $a_n \notin \R$.  In particular, this construction yields
uncountably many Pauli nonunique functions with Pauli partners differing
not just by a trivial ambiguity in the classical sense.

Furthermore for each $b \geq 3$, this construction yields an
orthogonal decomposition of
$L^2(\R) = \bigoplus_{c=0}^{b-1} \Hf_{b,c}$ such that every
$f \in \Hf_{b,c}$ is Pauli nonunique.  By a tensor product argument,
this can be generalized to higher
dimensions~\cite{MR0489489}.


\begin{remark}
  See also Ismagilov~\cite{MR1402086}, Janssen~\cite{MR1146076}, and
  Jaming~\cite{Jaming1999} for another construction of uncountably
  many Pauli partners which are not trivial solutions of the classical
  phase retrieval problem.
\end{remark}



Conversely, there is also a big class of functions where the Pauli
problem has a unique solution.  Friedman~\cite{MR976492} proved that
any nonnegative function is Pauli unique.



A generalized version of the Pauli problem was considered by
Jaming~\cite{jaming14uniqueness} for the fractional Fourier transform.  It is defined
for $f \in L^1(\R) \cap L^2(\R)$ by
\begin{equation*}
  \F_{\alpha}f (\xi) = c_{\alpha} e^{- \pi i  \cot(\alpha) |\xi|^2} 
  \F(e^{- \pi i \cot(\alpha) |\, .\,|^2}f)(\xi / \sin(\alpha))
\end{equation*}
with respect to the angle $\alpha \notin \pi \Z$ and where $c_{\alpha}$ is a
normalization constant such that $\F_{\alpha}$ is an isometry on
$L^2(\R)$. Note that $\F_{\pi/2}f = \hat{f}$ and $\F_0f = f$,
$\F_{\pi}f = R f$ by a limit procedure~\cite{330368}.





In terms of the fractional Fourier transform, the original Pauli
problem asks if a function $f \in L^2(\R)$ is uniquely determined by
the measurements $\{|\F_0f|, |\F_{\pi/2}f|\}$.  The natural
generalization is the following phase retrieval problem.

\begin{problem}[extended Pauli problem]
  Suppose $\tau \subseteq [-\pi/2,\pi/2]$ is a given set of angles
  (not necessarily finite).  Does the set of fractional Fourier
  measurements $\{ |\F_{\alpha}f| : \alpha \in \tau\}$ uniquely
  determine $f \in L^2(\R)$?
\end{problem}

Let us first discuss the case where $\tau$ consists of only one angle.
In this case, the proof of Theorem~\ref{thm:ambigcontinuous} can be
generalized to the fractional Fourier
transform~\cite{jaming14uniqueness}.  Hence compactly supported
functions in $L^2(\R)$ are not uniquely determined by any single
fractional Fourier measurement by a ``zero-flipping'' argument.

On the other hand, taking ``sufficiently
dense'' fractional Fourier measurements guarantees uniqueness in the
extended Pauli problem.



\begin{theorem}[{\cite[Theorem 5.1]{jaming14uniqueness}}]
  \label{sec:pauli-problem-extendedThm}
  Let $f, g \in L^2(\R)$, $\tau \subseteq [-\pi/2, \pi/2]$, and
  $|\F_{\alpha}f| = |\F_{\alpha}g|$ for all $\alpha \in \tau$.  Then
  the following hold:
  \begin{enumerate}
  \item If $\tau = [-\pi/2, \pi/2]$, then
    there exists a constant $c \in \Cf$ with $|c|=1$ such that $f = c g$.
  \item If $f,g$ have compact support, and $\tau$ is of positive measure
    or has an accumulation point $\alpha_0 \neq 0$, then
    there exists a constant $c \in \Cf$ with $|c|=1$ such that $f = c g$.
  \item If the support of $f,g$ is included in $[-a,a]$ and
    $\tau:=\{\pi / 2\} \cup \{\arctan a^2/k : k \in \Z \setminus
    \{0\}\}$, then there exists a constant $c \in \Cf$ with $|c|=1$
    such that $f = c g$.
  \end{enumerate}
\end{theorem}

The proof of Theorem~\ref{sec:pauli-problem-extendedThm} relies on a
relation between the fractional Fourier transform and the ambiguity
function
\begin{equation*}
  Af(x,\xi) := \int_{\R}f\Big(t+\frac{x}{2}\Big) \overline{f\Big(t-\frac{x}{2}\Big)}
  e^{-2\pi i t \cdot \xi} \dif t \,.
\end{equation*}
More precisely, one can show that (see \cite{330368})
\begin{equation*}
    \F(|\F_{\alpha}f|^2)(\xi) = A(\F_{\alpha}f, \F_{\alpha}f)(0,\xi) = Af(- \xi \sin(\alpha),
    \xi \cos(\alpha))\,.
\end{equation*}
Therefore, knowledge of $|\F_{\alpha}f|$ for a particular angle
$\alpha \in [-\pi/2, \pi/2]$ translates to knowing the values of the
ambiguity function $Af$ on a line in the time-frequency plane.  Since
$ Af(x, \xi) = \F ( T_{-x/2}f \cdot \overline{T_{x/2}f}) (\xi)$, one
can easily recover $f$, up to a global phase factor, from $Af$ by
taking the inverse Fourier transform (see, for example,
\cite{auslander85radar,wilcox1960synthesis} or the textbook
\cite{MR1843717}).

This leads to (i) immediately.  Statement (ii) requires a brief
excursion into complex analysis: Due to the compact support of $f$,
its ambiguity function
$ Af(x, \,.\,)=\F ( T_{-x/2}f \cdot \overline{T_{x/2}f})$ is an entire
function for every fixed $x \in \R$. Hence it is already uniquely
determined on a set with accumulation point.  For (iii) one employs
the Shannon--Whittaker formula for band-limited functions, where the
angles $\alpha_k$ are chosen precisely to correspond to the samples.

The sufficient conditions of
Theorem~\ref{sec:pauli-problem-extendedThm} require at least countably
many fractional Fourier measurements for uniqueness.  A natural
question is to ask, whether only finitely many would suffice.
Jaming~\cite{jaming14uniqueness} showed that functions of a specific
structure, like pulse-train signals or linear combinations of
Gaussians or Hermite functions, require only one or two fractional
Fourier measurements to be uniquely determined within their specific
type (but not necessarily with respect to all $L^2$ functions).

On the other hand, Andreys and Jaming~\cite{MR3552875} showed that any
finite set of angles $\tau = \{\alpha_1, \dots, \alpha_N\}$ with
$\cot(\alpha_k) \in \Qf$ for all $k=1,\dots, N$ is not sufficient for
uniqueness in the generalized Pauli problem.  Their result extends the
methods of Janssen~\cite{MR1146076} for the classical setting.

\subsubsection{Ambiguity Phase retrieval}

We continue with a phase retrieval problem for the ambiguity
function, which appears in radar
theory~\cite{auslander85radar,wilcox1960synthesis}.  Recall that the
\emph{ambiguity function} is defined for $f \in L^2(\R)$ by
\begin{equation*}
  Af(x,\xi) := \int_{\R}f\Big(t+\frac{x}{2}\Big) \overline{f\Big(t-\frac{x}{2}\Big)}
  e^{-2\pi i t \cdot \xi} \dif t \qquad \forall x, \xi \in \R \,.
\end{equation*}

The \emph{(narrow band) radar ambiguity problem} is now formulated as follows.

\begin{problem}[Radar Ambiguity Problem]
  \label{sec:ambig-phase-retr-prob}
  Does the modulus the ambiguity function $|Af|$ determine $f \in L^2(\R)$
  uniquely?
\end{problem}

Again, we will say that two functions $f,g \in L^2(\R)$ are
\emph{ambiguity partners} if $|Af| = |Ag|$.

Recall the translation, modulation, and reflection operators
$T_\tau f(x) = f(x-\tau)$,
$M_\omega f(x) = e^{2 \pi \omega \cdot x}f(x)$, and $Rf(x) = f(-x)$,
respectively.  Then it is easy to see from the definition the following
trivial ambiguities of Problem~\ref{sec:ambig-phase-retr-prob}.

\begin{proposition}
  Let $f \in L^2(\Rd)$. Then each of the following choices of $g$ yields
  $|Af|=|Ag|$:
  \begin{enumerate}
  \item $g=cf$ for $|c|=1$;
  \item $g=T_\tau f$ for $\tau\in \R^d$;
  \item $g=M_\omega f$ for $\omega\in \R^d$;
  \item $g=Rf$.
  \end{enumerate}
\end{proposition}

A first example of nontrivial ambiguity partners came from de
Buda~\cite{1054428}.  A systematic approach to studying the ambiguities
of Problem~\ref{sec:ambig-phase-retr-prob} can be found in
\cite{Jaming1999}.

For compactly supported functions $f \mkern-4mu \in \mkern-4mu L^2(\R)$, the ambiguity
function $Af(x, \,.\,) = \F ( T_{-x/2}f \cdot \overline{T_{x/2}f})$ is
an entire function in the second variable by the Paley--Wiener theorem.
Therefore $|Af(x,\xi)| = |Ag(x,\xi)|$ for $x, \xi \in \R$ is equivalent to
\begin{equation*}
  Af(x,z)  \overline{Af(x,\overline{z})} = Ag(x,z)  \overline{Ag(x,\overline{z})}
  \qquad \forall x \in \R, z \in \Cf \,.
\end{equation*}

Hence the ``zero-flipping'' that creates a lot of the ambiguities in
the classical Fourier phase retrieval problem may also appear for the
ambiguity function.  Unfortunately, zero-flipping is not well
understood for the ambiguity function.  Indeed, flipping some zeros of
$Af$ may not even yield an ambiguity function.

Jaming~\cite{Jaming1999} characterized the ambiguities of
Problem~\ref{sec:ambig-phase-retr-prob} excluding zero-flipping.  He
called two functions $f,g \in L^2(\R)$ with compact support
\emph{restricted ambiguity partners} if $Af(x, \,.\,)$ and
$Ag(x, \,.\,)$ have the same zeros in the complex plane and proved the
following result.


  

\begin{theorem}[{\cite[Theorem 4]{Jaming1999}}]
  Suppose $f \in L^2(\R)$ is a compactly supported function and
  let $\Omega$ be the open set of all $x$ such that $Af(x, \,.\,)$ is
  not identically $0$.

  Then  $g \in L^2(\R)$ is a restricted ambiguity partner of $f$ if and only if 
  there exists a locally constant function $\phi$
  on $\Omega$ such that, for every $t_0, t_1, t_2 \in \supp f$,
  \begin{equation*}
    \phi(t_2-t_1) + \phi(t_1 - t_0) \equiv \phi(t_2-t_0) \quad \operatorname{mod} 2\pi
  \end{equation*}
    and
  \begin{equation*}
    g(x) = c e^{i \phi(x-a-x_0)} e^{i \xi x} f(x-a)
  \end{equation*}
  for some $a, \xi \in \R$ and $|c|=1$.
\end{theorem}

Bonami, Garrig\'os, and Jaming~\cite{MR2362409} proved a uniqueness results for
\emph{Hermite functions}, \ie functions of the form
$f(x):= P(x) e^{-x^2/2}$, where $P$ is a polynomial.  Their proofs were
inspired by some preliminary results obtained in the 1970s by
Bueckner~\cite{bueckner1967signals} and de Buda~\cite{1054428}.

\begin{theorem}[{\cite[Theorem A]{MR2362409}}]
  \label{sec:ambig-phase-retr-Hermite}
  For almost all polynomials $P$, the Hermite function
  $f(x):= P(x) e^{-x^2/2}$ has only trivial ambiguity partners.
\end{theorem}

Here ``almost all'' is understood in the sense of the Lebesgue measure
after identifying the space of $n$-dimensional polynomials with
$\Cf^{n+1}$.

The authors of \cite{MR2362409} note that the ``almost all'' part in
Theorem~\ref{sec:ambig-phase-retr-Hermite} may be an artifact of the
proof and strongly believe that all Hermite functions have only
trivial ambiguities.  Furthermore, Jaming~\cite{Jaming1999}
conjectured that similar results hold for functions of the form
$f(x)= P(x) e^{-x^2/2}$ with $P$ an entire function of order
$\alpha < 1$, but the techniques of \cite{MR2362409} do not apply in
this case.

\begin{conj}
  \leavevmode
  \begin{enumerate}
  \item If $f$ is a Hermite function, \ie $f(x)= P(x) e^{-x^2/2}$ with
    a polynomial $P$, then $f$ only has trivial ambiguity partners.
    (Bueckner~\emph{\cite{bueckner1967signals}})
  \item If $f(x)= P(x) e^{-x^2/2}$ with $P$ an entire function of
    order $\alpha < 1$, then f has only trivial ambiguity
    partners. (Jaming~\emph{\cite{Jaming1999}})
  \end{enumerate}
\end{conj}

We end this section by stating that we only considered the
``narrow-band'' ambiguity problem, where certain physical
restrictions are assumed of the signal.  If those assumptions are
lifted, the physical measurements yield the \emph{wide-band ambiguity
  function}, which is related to the wavelet transform.  The
``wide-band'' ambiguity phase retrieval problem is even less
understood. We refer the interested reader to \cite{Jaming1999} for
some phase retrieval results in this direction.



\subsubsection{Continuous Short-Time Fourier Transform Phase Retrieval Problem}

We finally turn to the continuous \stft\ phase retrieval.  Recall that
the \emph{\stft} (STFT) of $f \in L^2(\Rd)$ with respect to the window
$g \in L^2(\Rd)$ is defined by
\begin{equation*}
  V_gf(x,\xi):= (f \cdot T_x \bar{g})\Fhat (\xi) =
  \int_{\Rd} f(t) \overline{g(t-x)} e^{-2\pi i t \cdot \xi } \dif x
  \qquad \forall x,\xi \in \Rd \,.
\end{equation*}


If we fix the window $g$, the \stft\ $V_g$ is a linear operator from
$L^2(\Rd)$ to $L^2(\Rtd)$.  Consequently, a multiplication of $f$ with
a unimodular constant produces the same phaseless \stft\ measurements
and is therefore considered a trivial ambiguity.  The problem of phase
retrieval now reads as follows.

\begin{problem}[short-time Fourier phase retrieval]
  \label{prob:stft}
  Suppose $f\in L^2(\Rd)$. Recover $f$ from $|V_g f|$ up to a global
  phase factor when $g \in L^2(\Rd)$ is known.
\end{problem}

Whether Problem \ref{prob:stft} is well-posed depends on the choice of
the window $g$.  Again a sufficient condition for uniqueness is given
in terms of the zero set of its \stft\ $V_gg$.  The proof of this
result is analogous to the discrete case with the following
fundamental formula at its core.



\begin{proposition}
  \label{sec:cont-stft-phase-MagicC}
  Let $f,h, g, u \in L^2(\Rd)$.  Then 
  \begin{equation*}\label{eq:MagicC}
    (V_gf \cdot \overline{V_{u}h})\Fhat(x, \xi) 
    =  (V_hf \cdot \overline{V_{u}g})(-\xi, x)
    \qquad \forall x,\xi \in \Rd  \,.
  \end{equation*}
\end{proposition}

Proposition~\ref{sec:cont-stft-phase-MagicC} is obtained as in the
discrete setting by combining the covariance property with the
orthogonality relations.  The relevant properties of the
\stft\ and their detailed proof can be found in
\cite{MR1843717,MR3232589}.\\

We can now prove the following theorem.
\begin{theorem}
  \label{thm:stftuniqueness}
  Let $g\in L^2(\Rd)$ with $V_gg(x,\xi) \neq 0$ for almost all
  $x,\xi \in \Rd$.  Then for any $f,h\in L^2(\Rd)$ with
  $|V_g f|=|V_g h|$ there exists $\alpha \in \R$ such that
  $h=e^{i\alpha} f$.
%
\end{theorem}

More general versions of this statement can be found in
\cite{grohs2017stable} and \cite{luef2018generalized}.

\begin{proof}
  By Proposition~\ref{sec:cont-stft-phase-MagicC} we obtain that
  \begin{equation*}
    \label{eq:magicS}
    (|V_gf|^2)\Fhat(x,\xi) =  V_ff(-\xi,x) \cdot \overline{V_gg(-\xi,x)} 
    \qquad \forall x,\xi \in \Rd 
  \end{equation*}
  and recover $V_ff$ almost everywhere.  It is easy to see that
  $V_f f$ uniquely determines $f$ up to a phase factor by taking the
  inverse Fourier transform.
\end{proof}

\begin{remark}
  The fact that $V_f f$ uniquely determines $f$ up to a phase factor
  is now a standard result in time-frequency analysis. See, for example,
  \cite{auslander85radar,wilcox1960synthesis} or the textbook
  \cite{MR1843717}.
\end{remark}

Let us mention some examples for window functions that allow phase
retrieval because their \stft\ does not vanish.  The obvious candidate
is the Gaussian $\phi(x)= e^{- \pi |x|^2}$, whose \stft\
$V_{\phi}\phi$ is again a (generalized) Gaussian.  A lesser known
example is the one-sided exponential
$g(x) = e^{-\alpha x} \chi_{[0,\infty)}$ for parameter $\alpha >0$.
Already Janssen~\cite{MR1621312} computed its \stft\
$V_gg = e^{-|x|(\alpha + \pi i \xi)}/(2 \alpha + 2 \pi i \xi)$, which
clearly does not vanish.  More examples can be found in the recent
paper by Gr\"ochenig, Jaming, and Malinnikova~\cite{Groechenig2018}.

The choice of the one-dimensional Gaussian $\phi(x) = e^{-\pi |x|^2}$
is special in one crucial point: it is the only window for which
$V_{\phi}f$ yields a holomorphic function after a slight
modification~\cite{MR2729877}.  Hence the full toolbox of complex
analysis becomes available when working with a Gaussian window.  This
modified transform is best known as the \emph{Bargmann} transform.



In the remainder, we present a result of two of the authors
\cite{grohs2017stable} which gives a characterization of
instabilities of the short-time Fourier phase retrieval problem with
Gaussian window.  The work in \cite{grohs2017stable} builds upon
results by one of the authors and his collaborators
\cite{alaifari2016stable}, where for phaseless measurements arising
from holomorphic functions
it is shown that the phase can be stably recovered on so-called atolls.

By an instability we mean, roughly speaking, a signal $f$ for which
there exists a signal $g$ which is very different from $f$, but at the
same time produces very similar phaseless measurements.  This
intuition is formalized by the local Lipschitz constant of the
solution operator $|V_{\phi}f|\mapsto f\sim e^{i\alpha}f$.
\begin{definition}\label{def:loclip}
  Let $\mathcal{A}$ be a mapping from $\mathcal{X}$ to $\mathcal{Y}$,
  where $(\mathcal{X},d_{\mathcal{X}})$ and
  $(\mathcal{Y},d_{\mathcal{Y}})$ are metric spaces.  Then the
  \emph{local stability constant} $C_{\mathcal{A}}(f)$ of
  $\mathcal{A}$ at $f\in \mathcal{X}$ is defined as the smallest
  positive number $C$ such that
 \begin{equation*}
  d_\mathcal{X}(f,g) \le C\cdot d_\mathcal{Y}(\mathcal{A}f,\mathcal{A}g) \quad \forall g\in \mathcal{X}.
 \end{equation*}
\end{definition}
Instabilities are routinely constructed by fixing a well-localized
function $f_0$; then for large $\tau$ the functions
$$
f_\pm^\tau:=f_0(\cdot-\tau) \pm f_0(\cdot +\tau)
$$
yield approximately the same phaseless short-time Fourier measurements.  Even more
so the stability constant degenerates exponentially in $\tau$, i.e.,
$C_{|V_{\phi}|}(f_+^\tau)\gtrsim e^{c \tau^2}$ for suitable metrics
\cite{alaifari2018gabor}.

As we shall see, the stability constant for short-time Fourier phase
retrieval with Gaussian window can be controlled in terms of a concept
which was introduced by Cheeger in the field of Riemannian
geometry \cite{cheeger1969lower}.
\begin{definition}
  Let $\Omega\subseteq \Rd$ be open.  For a continuous, nonnegative,
  integrable function $w$ on $\Omega$ the \emph{Cheeger constant} is
  defined as
 \begin{equation}\label{def:cheegerconst}
 h(w,\Omega):=\inf_{C\subseteq \Omega \atop \partial C \text{ smooth}} \frac{\int_{\partial C \cap \Omega} w}{\min \{\int_C w, \int_{\Omega\setminus C} w\}}.
 \end{equation}
\end{definition}
A small Cheeger constant indicates that the domain can be partitioned
into two subdomains such that the weight is rather small on the
separating boundary of the two subdomains and that, at the same time
both subdomains carry approximately the same amount of $L^1$-energy.
\begin{figure}
 \begin{subfigure}{0.32\textwidth}
  \includegraphics[width=\textwidth]{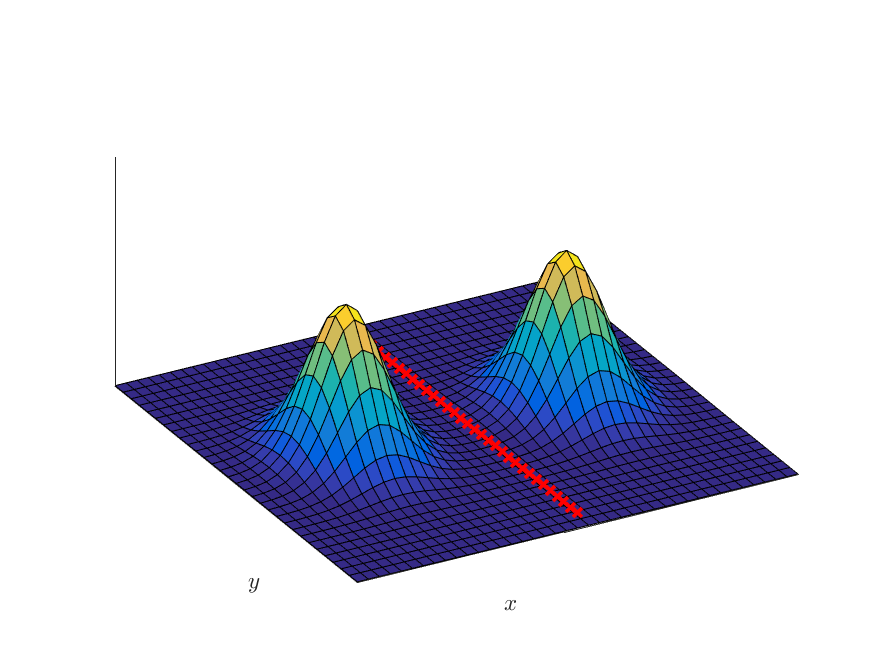}
 \end{subfigure}
 \begin{subfigure}{0.32\textwidth}
  \includegraphics[width=\textwidth]{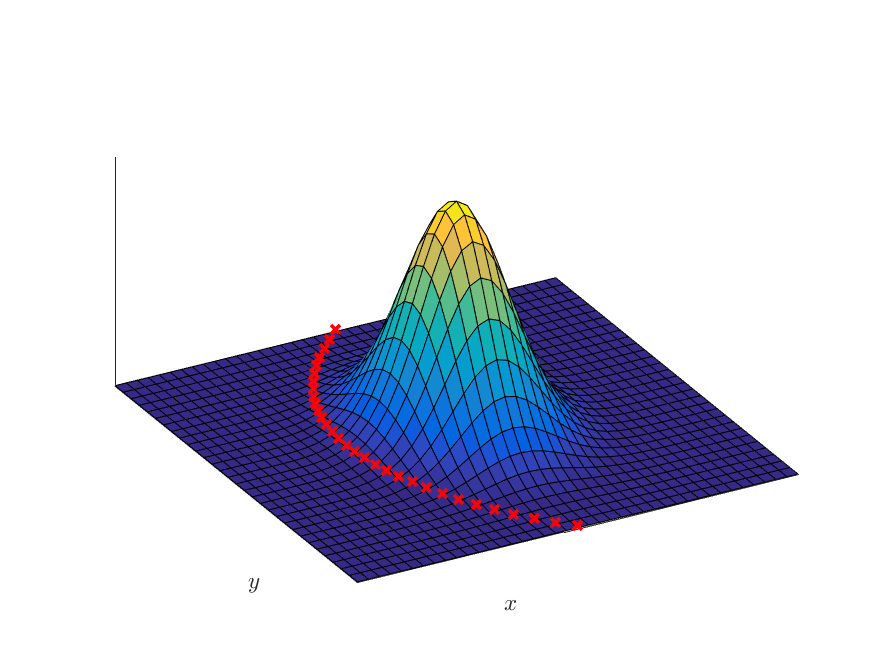}
 \end{subfigure}
  \begin{subfigure}{0.32\textwidth}
  \includegraphics[width=\textwidth]{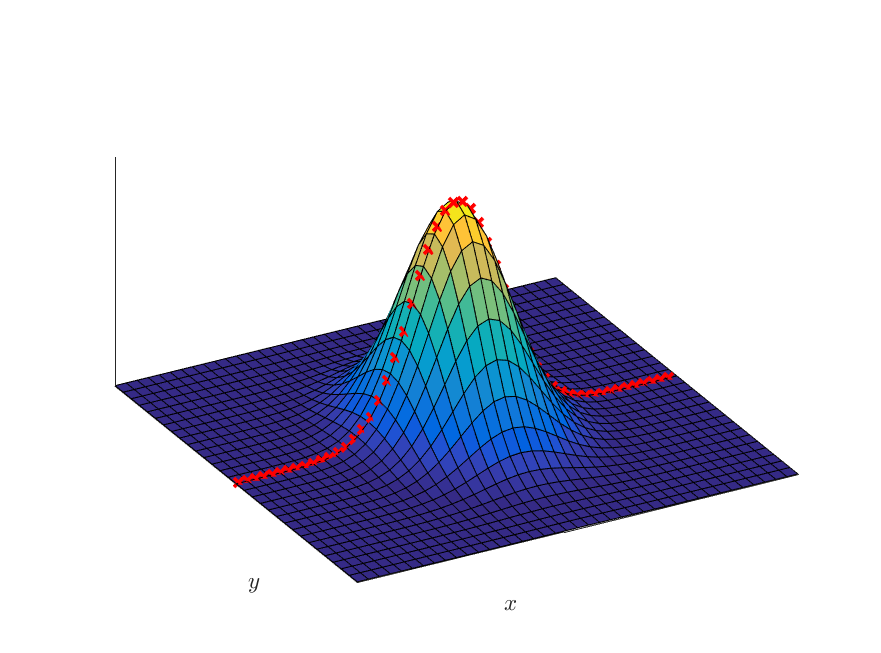}
 \end{subfigure}
 \caption{If the weight has its mass concentrated on two or more disjoint subdomains a partition can be found such that both components of the partition carry approximately the same amount of energy and at the same time 
 the weight is small along the separating boundary (left figure), i.e., the Cheeger constant is small in that case.\\
 If on the other hand the mass is concentrated on a single connected domain a partition which satisfies both requirements cannot be found: 
 Aiming for small values of the weight along the separating boundary will not distribute the mass well between the two components (center), 
 whereas a fair distribution of the mass entails that the weight is substantially large on parts of the boundary.}
 \label{fig:cheeger}
\end{figure}
In that sense the Cheeger constant captures the disconnectedness of the weight; cf. Figure~\ref{fig:cheeger}.

Before we state the stability result, both the signal space and the measurement space have to be endowed with suitable metrics.
To this end we define Feichtinger's algebra and a family of weighted Sobolev norms.
\begin{definition}
\emph{Feichtinger's algebra} is defined as 
 $$
 \mathcal{M}^1:=\{f\in L^2(\R): ~V_{\phi}f\in L^1(\R^2)\},
 $$
 with induced norm
 $
 \|f\|_{\mathcal{M}^1}:=\|V_{\phi}f\|_{L^1(\R^2)}.
 $
\end{definition}
\begin{definition}
 For $1\le p,q< \infty$, $s>0$ and $F:\R^2\rightarrow \Cf$ sufficiently smooth we define
 $$
 \|F\|_{\mathcal{D}_{p,q}^s} := \|F\|_{L^p(\R^2)} + \|\nabla F\|_{L^p(\R^2)} + \|F\|_{L^q(\R^2)} + \|(|x|+|y|)^s F(x,y)\|_{L^q(\R^2)}. 
 $$
\end{definition}
The main stability result in \cite{grohs2017stable} now reads as follows.
\begin{theorem}\label{thm:gaborstab}
  Let $q>2$.  Let $\mathcal{X}:=\mathcal{M}^1/ \sim$ be endowed with
  the metric\footnote{$f \sim g$ if and only if $g=e^{i\alpha} f$ for
    some $\alpha \in \R$.}
$$
d_\mathcal{X}([f]_\sim,[g]_\sim) := \inf_{\alpha \in \R} \|f-e^{i\alpha} g\|_{\mathcal{M}^1},
$$
and let $\mathcal{Y}:=|V_{\phi}|(\mathcal{M}^1)$ be endowed with the
metric induced by the norm $\|\cdot\|_{\mathcal{D}_{1,q}^4}$.  Suppose
that $f\in \mathcal{M}^1$ is such that $|V_{\phi}f|$ has a global
maximum at the origin.  Then there exists a constant $c$ that only
depends on $q$ and the quotient
$\|f\|_{\mathcal{M}^1} / \|V_{\phi}f\|_{L^\infty(\R^2)}$ such that
\begin{equation}\label{est:stabconst}
C_{|V_{\phi}|}(f) \le c (1+h(|V_{\phi}f|,\R^2)^{-1}).
\end{equation}
\end{theorem}
Disregarding the weak dependence of $c$ on $f$ the estimate
\eqref{est:stabconst} can be informally summarized as follows:
\begin{quotation}
  The only instabilities for short-time Fourier phase retrieval with
  Gaussian window are of disconnected type.
\end{quotation}
Before we give a sketch of the proof we set the stability result in
relation to the general results in the abstract setting in
section~\ref{sec:stability}, where the concept of the $\sigma$-strong
complement property was introduced.  In the context of short-time Fourier phase
retrieval Remark \ref{rem:locstab} can be qualitatively understood in
the following way.  A function $x$ is rather unstable if it can be
written as $x=f+h$ with $\|f\|_{L^2(\R)}, \|h\|_{L^2(\R)} \asymp 1$
such that their respective short-time Fourier measurements are essentially
supported on two disjoint domains. In other words the time-frequency
plane can be split up into $S\subseteq \R^2$ and $\R^2\setminus S$
such that both $\|V_{\phi}f\|_{L^2(S)}$ and
$\|V_{\phi}h\|_{L^2(\R^2\setminus S)}$ are small.  If the metrics on
the signal and measurement space are both induced by the respective
$L^2$-norm it holds that
\begin{equation}\label{est:lowerboundstabc}
C_{|V_{\phi}|}(x)\gtrsim \sup_{f,h:x=f+h \atop S\subseteq \R^2} \frac{\max\{ \|V_{\phi}f\|_{L^2(S)},\|V_{\phi}h\|_{L^2(\R^2\setminus S)}\}}{\min\{\|f\|_{L^2(\R)}, \|h\|_{L^2(\R)}\}}.
\end{equation}
Theorem \ref{thm:gaborstab} nicely complements this result in the
sense that the disconnectedness as quantified by the Cheeger
constant---which to some extent resembles the lower bound in
\eqref{est:lowerboundstabc}---also gives an upper bound on the local
stability constant.

\begin{proof}[Architecture of the proof]
Let us start with the observation that for any $f,g\in \mathcal{M}^1$ it holds that 
\begin{equation}\label{eq:proofstable1}
 d_\mathcal{X}([f]_\sim,[g]_\sim) =\inf_{|c|=1} \left\| \frac{V_{\phi}g}{V_{\phi}f} - c\right\|_{L^1(\R^2,w)},
\end{equation}
where $w=|V_{\phi}f|$.

Now suppose that we could just disregard the constraint $|c|=1$ in \eqref{eq:proofstable1} (this can be justified with considerable effort).
The Poincar\'e inequality tells us that there exists a constant $C_{poinc}(w)$ such that \eqref{eq:proofstable1} can be bounded by
\begin{equation}\label{eq:proofstable2}
 C_{poinc} (w) \cdot \left\|\nabla \frac{V_{\phi}g}{V_{\phi}f}\right\|_{L^1(\R^2,w)}.
\end{equation}

Now spectral geometry enters the picture.  Cheeger's inequality
\cite{cheeger1969lower} says that the Poincar\'e constant on a
Riemannian manifold can be controlled by the reciprocal of the Cheeger
constant.  We would like to apply this result to the metric induced by
the metric tensor
$\left(w(x,y)\begin{bmatrix}1 & 0 \\ 0 &
    1\end{bmatrix}\right)_{(x,y)\in \R^2}$ in order to get a bound on
$ C_{poinc}(w)$.  However, since $w$ in our case arises from short-time Fourier
measurements it generally has zeros and therefore does not qualify as a
Riemannian manifold.  Nevertheless a version of Cheeger's inequality
can be established, i.e.,
\begin{equation}\label{eq:proofstable3}
C_{poinc}(w) \lesssim h(w,\R^2)^{-1},
\end{equation}
where $h(w,\R^2)$ is defined as in \eqref{def:cheegerconst}.

Next we will make use of the fact that for any $h\in L^2(\R)$
\begin{equation}\label{eq:gaborholo}
z=x+iy \mapsto V_{\phi}h(x,y)\cdot e^{\pi(|z|^2/2-ixy)}
\end{equation}
is an entire function (up to reflection).  Thus $V_{\phi}g/V_{\phi}f$
is meromorphic (again up to reflection) and by applying the
Cauchy--Riemann equations one elementarily computes that
\begin{equation}\label{eq:proofstable4}
 \left| \nabla \frac{V_{\phi}g}{V_{\phi}f}\right| = \sqrt2 \cdot \left| \nabla \left|\frac{V_{\phi}g}{V_{\phi}f}\right|\right|
\end{equation}
almost everywhere. Combining \eqref{eq:proofstable1}, \eqref{eq:proofstable2}, \eqref{eq:proofstable3} and \eqref{eq:proofstable4}
yields that 
\begin{equation*}\label{eq:proofstable5}
 d_\mathcal{X}([f]_\sim,[g]_\sim) \lesssim h(w,\R^2)^{-1} \cdot \left\| \nabla \left|\frac{V_{\phi}g}{V_{\phi}f}\right| \right\|_{L^1(\R^2,w)}\,.
\end{equation*}
This means that we already succeeded in bounding the distance between
the signals in terms of their phaseless short-time Fourier
measurements.  The aim, however, is to get a bound in terms of the
difference of the \stft\ magnitudes.  In order to obtain
this, we estimate
\begin{equation*}\label{eq:proofstable6}
  \left\| \nabla \left|\frac{V_{\phi}g}{V_{\phi}f}\right| \right\|_{L^1(\R^2,w)} \le
  \left\| \left(\frac{\nabla |V_{\phi}f|}{|V_{\phi}f|}\right) \left(|V_{\phi}f|-|V_{\phi}g|\right)\right\|_{L^1(\R^2)}
  + \left\| \nabla |V_{\phi}f| - \nabla |V_{\phi}g|\right\|_{L^1(\R^2)}.
\end{equation*}

The final ingredient of the proof lies in the treatment of the
logarithmic derivative $\frac{\nabla |V_{\phi}f|}{|V_{\phi}f|}$.  The
norm of the logarithmic derivative on balls centered at the origin can
essentially be controlled by the product of the volume of the ball and
the number of its singularities in a ball of twice the radius, which
are precisely the zeros of $V_{\phi}f$.  Jensen's formula relates the
number of zeros of the function in \eqref{eq:gaborholo}, and therefore
of $V_{\phi}f$, to its growth.  Since the growth of the entire
functions in \eqref{eq:gaborholo} can be uniformly bounded for
functions $f\in \mathcal{M}^1$ this argument allows us to absorb the
logarithmic derivative in a lower order polynomial, which is
independent of $f$.
\end{proof}


\subsection*{Acknowledgments}
The authors thank Martin Ehler for reading and commenting on parts of
the manuscript.  Furthermore, the authors highly appreciated the
constructive feedback of both referees, which considerably improved
this work.  Finally, the last two authors graciously acknowledge the
support of the Austrian Science Fund (FWF): S.K.\ was employed in
the project P30148-N32, and M.R.\ was supported by the START-Project
Y963-N35.

\bibliographystyle{abbrv}
\bibliography{phase}

\end{document}